\documentclass[11pt]{amsart}

\usepackage[
    a4paper,
    bindingoffset=0.2in,
    left=0.8in,
    right=1in,
    top=1in,
    bottom=1in,
    footskip=.25in
]{geometry}

\usepackage{amsmath}
\usepackage{thmtools}
\usepackage{lmodern}
\usepackage{microtype}
\usepackage[dvipsnames,svgnames,x11names]{xcolor}
\usepackage{pgfplots}
\usepgfplotslibrary{groupplots}
\usetikzlibrary{arrows.meta,calc}
\pgfplotsset{compat=1.18}

\usepackage{mathtools,amssymb,bm}
\usepackage[mathscr]{euscript}

\usepackage{enumitem}
\usepackage{comment}

\usepackage[
    pagebackref,
    colorlinks,
    citecolor=Mahogany,
    linkcolor=Mahogany,
    urlcolor=Mahogany,
    filecolor=Mahogany
]{hyperref}
\usepackage[capitalize]{cleveref}

\DeclareFontFamily{U}{mathx}{}
\DeclareFontShape{U}{mathx}{m}{n}{<-> mathx10}{}
\DeclareSymbolFont{mathx}{U}{mathx}{m}{n}
\DeclareMathAccent{\widecheck}{0}{mathx}{"71}

\newtheorem{theorem}{Theorem}[section]
\newtheorem{lemma}[theorem]{Lemma}
\newtheorem{proposition}[theorem]{Proposition}
\newtheorem{corollary}[theorem]{Corollary}

\newtheorem{atheorem}{Theorem}

\theoremstyle{definition}
\newtheorem{definition}[theorem]{Definition}

\theoremstyle{remark}
\newtheorem{remark}[theorem]{Remark}
\newtheorem{question}[theorem]{Question}

\numberwithin{equation}{section}
\crefname{equation}{}{}

\newcommand{\sq}[1]{\widetilde{#1}}
\newcommand{\R}{\mathbb{R}}

\newcommand{\C}{\mathbb{C}}

\newcommand{\mc}[1]{\mathcal{#1}}

\newcommand{\ov}[1]{\overline{#1}}
\newcommand{\ang}[1]{\langle #1 \rangle}

\setlist[itemize,1]{label=$\cdot$}

\title{From weighted paraboloid restriction to $k$-stars and distance graphs
}
\author[T. Borges, Y. Ou, M. Pasquariello]{Tainara Borges, Yumeng Ou, Marcus Pasquariello}

\address[T. Borges]{Department of Mathematics, University of Pennsylvania, Philadelphia, PA 19104, USA}\email{tborges@sas.upenn.edu}

\address[Y. Ou]{Department of Mathematics, University of Pennsylvania, Philadelphia, PA 19104, USA}\email{yumengou@sas.upenn.edu}

\address[M. Pasquariello]{Department of Mathematics, Brown University, Providence, RI 02912, USA}\email{marcus\_pasquariello@brown.edu}

\date{\today}

\begin{document}

\begin{abstract}
In this paper, we study pinned $k$-star distance sets associated to compact subsets of $\mathbb{R}^n$, $n\geq 2$. For pins $x_1,\dots,x_k\in E$, the pinned $k$-star distance set is
\[
\Delta_{x_1,\dots,x_k}^{k\text{-star}}(E)
=
\{(|x_1-x|,\dots,|x_k-x|):x\in E\}\subset\mathbb{R}^k.
\]

We obtain improved Hausdorff-dimension thresholds on $E$ guaranteeing that pinned $k$-star distance sets have positive $k$-dimensional Lebesgue measure. The main analytic input is a reformulation of the connection, first observed in \cite{IPPS22}, between $k$-stars in $\mathbb{R}^n$ and pinned dot products on the paraboloid in $\mathbb{R}^{n+1}$. In our framework, \(L^2(\R^k)\) estimates for the densities of pinned \(k\)-star distance measures are reduced to a weighted Fourier extension estimate for the paraboloid whose weight is defined explicitly in terms of Frostman measures on $E$. For $1\leq k<n$, this yields the threshold 
\[\dim(E)>\alpha_{+}(n,k):=\frac{n^2+nk+k}{2n+1}=\frac{n+k-1}{2}+\frac14
+\frac{2k+1}{4(2n+1)}.\]

Using the graph-building machinery of \cite{BFOPR2026}, our positive-measure results for $k$-stars can be used as building blocks for finite distance graph configurations with prescribed pins. As a consequence, we improve the best-known positive-measure thresholds for pinned $k$-simplices in every dimension $n\geq 3$ and for necklace graphs (cycles) in every dimension $n\geq 3$.

We further prove nonempty interior results for $k$-stars. In the special case $k=1$, corresponding to pinned nonempty interior of the distance set $\Delta_{x}(E)=\{|x-y|\colon y\in E\}$, we use a sharper argument to improve the pinned nonempty-interior thresholds of \cite{BFOP2026} in all dimensions $n\geq 4$.

\end{abstract}

\maketitle

\section{Introduction and main results}

A central theme in geometric measure theory is to understand when a fractal set
determines many geometric patterns. The prototypical example is the Falconer
distance problem: given a compact set \(E\subset \mathbb R^n\), $n\geq 2$, one asks how large
\(\dim(E)\) must be in order to ensure that
\[
\Delta(E)=\{|x-y|:x,y\in E\}
\]
has positive Lebesgue measure. A more refined version is the pinned distance
problem, where one fixes a pin $x\in E$ and studies
\[
\Delta_x(E)=\{|x-y|:y\in E\}.
\]
Here the goal is to find dimensional hypotheses guaranteeing that
\(\Delta_x(E)\) has positive measure for at least one, or for many, pins
\(x\in E\), which in particular implies positive measure for $\Delta(E)$. The Falconer distance set conjecture \cite{Falconer85} says that, for the distance set to have positive Lebesgue measure, it suffices to assume the Hausdorff dimension of $E\subset \mathbb{R}^n$ to be larger than $\frac{n}{2}$.

The Falconer distance problem has seen major progress in recent years. In the
plane, Guth, Iosevich, the second listed author, and Wang~\cite{GIOW20} proved that if
\(E\subset \mathbb R^2\) and \(\dim(E)>5/4\), then there exists \(x\in E\) such
that the one-dimensional Lebesgue measure of $\Delta_x(E)$ is positive:
\[
\mathcal L^1(\Delta_x(E))>0.
\]
In dimensions \(n\geq 3\), the best currently known threshold is due to Du, the second listed author,
Ren, and Zhang~\cite{DORZ2023}, who proved the same conclusion under the
condition
\[
\dim(E)>\frac n2+\frac14-\frac{1}{8n+4}.
\]Very recently, Liu \cite{liu2026} proved the Falconer distance conjecture in the plane under additional assumption of regularity of the set $E$.

One may also ask for the stronger conclusion that distance sets have nonempty interior. In the unpinned setting, Mattila and Sj\"olin~\cite{MS99} proved that
\[
\dim(E)>\frac{n+1}{2}
\]
implies \(\operatorname{int}(\Delta(E))\neq\emptyset\). The pinned nonempty-interior problem is more delicate, and the best known thresholds remain higher than their unpinned counterparts. Peres and Schlag~\cite{PS00} showed that, for $n\geq 3$, the condition
\[
\dim(E)>\frac{n+2}{2}
\]
guarantees the existence of \(x\in E\) such that \(\operatorname{int}(\Delta_x(E))\neq\emptyset\). This was recently improved by the first two authors of the present paper, together with Foster and Palsson~\cite{BFOP2026}: they proved that the threshold $7/4$ suffices in the plane, and that in dimension $n=3$ the Peres--Schlag threshold can be lowered from $5/2$ to $12/5$.

More generally, one may study distance vectors associated to finite graphs whose vertices are sampled from \(E\), with some vertices fixed in advance. In this paper, we focus on the pinned $k$-star configurations, which play a fundamental role in the theory of general distance graphs. First, it is a natural generalization of the distance set, as the $k=1$ case corresponds exactly to the aforementioned Falconer distance problem. Second, results of $k$-stars can be used as building blocks to derive corresponding estimates for any pinned graph configurations. We give a more detailed discussion of this application towards the end of the introduction. 

More precisely, by a $k$-star we mean the pinned graph with $k+1$ vertices $\{v_1,\dots ,v_k,v_{k+1}\}$ where the first $k$ vertices are pinned, and the edge set is $\{(v_i,v_{k+1})\colon 1\leq i\leq k\}$. See Figure \ref{fig: stargraph} for a $7$-star. The distance set associated with a pinned $k$-star is then 
\begin{equation}\label{def:kstar}
    \Delta^{k-star}_{x_1,x_2,\dots, x_k}(E):=\{(|x_1-x_{k+1}|,|x_2-x_{k+1}|,\dots, |x_k-x_{k+1}|)\colon x_{k+1}\in E\}.
\end{equation}

\begin{figure}[h]
    \centering
    \begin{tikzpicture}[scale=1.5, 
    every node/.style={circle, draw, minimum size=6pt, inner sep=0pt},
    label distance=1mm]

\node[fill=black, label=below:$v_8$] (C) at (0.2,0) {};

\foreach \i in {1,...,7} {
    \node[fill=magenta, label=above:$v_{\i}$] 
        (A\i) at ({cos(360/7*\i)}, {sin(360/7*\i)}) {};
    \draw (C) -- (A\i);
}

\end{tikzpicture}
    \caption{$S_7$, the $7$-star graph, with pins at each leaf.}
    \label{fig: stargraph}
\end{figure}
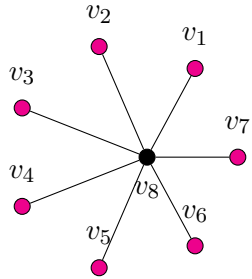

A simplified version of the main theorem of the paper is the following.

\begin{theorem}\label{thm: simple main}
Assume $n\geq 2$ and $1\leq k<n$. Let $E \subset \R^n$ be a compact set with Hausdorff dimension 
     \[\dim(E) > \alpha_{+}(n,k)=\frac{n^2 + nk + k}{2n + 1}=\frac{n + k - 1}{2} + \frac{1}{4} + \frac{2k + 1}{4(2n + 1)}.\]
     
Then there exist pins 
$x_1,\dots,x_k\in E$ such that
\[
\mathcal L^k\bigl(\Delta^{k\text{-star}}_{x_1,\dots,x_k}(E)\bigr)>0.
\]
\end{theorem}

\begin{remark}

Observe that $\alpha_{+}(n,k)=n$ when $n=k$. The condition $k<n$ cannot be removed, since it was observed in \cite{BFOPR2026} that a pinned $k$-star result with nontrivial threshold is only possible in $\R^n$ when $n>k$. Moreover, the result is mostly interesting for $k\geq 2$ and $n\geq 3$, since when $k=1$, the current thresholds for the Falconer distance problem are smaller than $\alpha_{+}(n,1)$ \cite{GIOW20,DORZ2023}.
\end{remark}

As far as we know, the study of the $k$-star configurations was initiated fairly recently in \cite{IPPS22}. And the main motivation was to understand the Falconer type problems for $k$-simplices. Very recently, in \cite{BFOPR2026}, it was further observed that $k$-stars can be used as building blocks to study any pinned distance graphs, which adds significant importance to the study of the $k$-star problem. In fact, the graph building mechanism in \cite{BFOPR2026} is quite delicate, and will need a stronger and more technical version of Theorem \ref{thm: simple main} as the building block. To state the full version of our main theorem, we first recall some definitions from \cite{BFOPR2026}. 

Given a simple graph $G$ with vertex set $\mathcal{V}$ and edge set $\mathcal{E}\subseteq 2^{\mathcal{V}}$, pick a set of vertices to pin, $\mathcal{P}=\{v_1,v_2,\dots ,v_m\}$ where $m\geq 0$ ($m=0$ corresponding to the unpinned case). Let $\mathcal{V}\setminus \mathcal{P}=\{v_{m+1},v_{m+2},\dots ,v_{|\mathcal{V}|}\}$ be the set of unpinned vertices. Then every edge can be identified with a pair $(i,j)$, some $1\leq i\neq j\leq |\mathcal{V}|$. We will further assume that no edge has been fully pinned, that is, $\mathcal{P}$  does not contain any pair of vertices $v_i,v_j$ such that $(i,j)\in \mathcal{E}$.

For distinct points $x_1,\dots,x_m\in \mathbb{R}^n$, 
we define the \emph{$m$-pinned $G$-set} of $E$ by
\begin{equation}\label{def: graphdistanceset}
\Delta^{G}_{\mathcal{P},x_1,\dots,x_m}(E)
:=\bigl\{ (|x_i-x_j|)_{(i,j)\in \mathcal{E}}
:\; x_{m+1},x_{m+2},\dots,x_{|\mathcal{V}|}\in E \text{ distinct}\bigr\}\subset \R^{|\mathcal{E}|}.
\end{equation}

In words, such a distance set captures vectors of edge lengths obtained by fixing the $m$ chosen pins in the graph at certain locations $x_1,x_2, \dots ,x_m$ in $E$, and by allowing the other unpinned vertices from the graph to be sampled anywhere in $E$. 

The notion of admissibility for pinned graphs was introduced in \cite{BFOPR2026}. We recall it here.

\begin{definition}[$k$-admissible pinned graphs] Let $(G,\mathcal{P})$ be a pinned graph with pins in $\mathcal{P}=\{v_i\}_{i=1}^{m}$. For $k\in\mathbb N$, a pinned graph $(G,\mathcal{P})$ is \emph{$k$-admissible} if there is a way of listing the unpinned vertices $v_{m+1},v_{m+2},\dots ,v_{|\mathcal{V}|}$ such that the back-degree is always bounded by $k$, that is,
\[
d_{<i}(v_i):=\big|\{\,u\in\{v_1,\dots,v_{i-1}\}:\ (u,v_i)\in\mathcal E\,\}\big|\le k,\,\forall i>m.
\]
\end{definition}

\begin{theorem}[\cite{BFOPR2026}]\label{thm: structuralthmforgraphs}
Let $G$ be a graph with vertex set $\mathcal{V}$ and edge set $\mathcal{E}$. Let $\mathcal{P}\subset \mathcal{V}$ be such that no pair of vertices in $\mathcal{P}$ is connected by an edge in $\mathcal{E}$ and denote $m=|\mathcal{P}|$. Assume that $(G,\mathcal{P})$ is a $k$-admissible pinned graph. Let $n\geq 2$. Then, for $1\leq k<n$ and any compact set $E\subset \mathbb{R}^n$ with Hausdorff dimension larger than $\frac{n+k}{2}$, there are distinct points $x_1,\cdots, x_m\in E$, such that the pinned distance set $\Delta^{G}_{\mathcal{P}, x_1,\cdots,x_m}(E)$ as defined in (\ref{def: graphdistanceset}) has positive $|\mathcal{E}|$-dimensional Lebesgue measure.

\end{theorem}

The threshold $\frac{n+k}{2}$ in the theorem above is intrinsically related to the previously best known threshold for $k$-stars obtained in \cite{IPPS22} and could be upgraded to the best thresholds for the pinned Falconer problem when $k=1$.

Before recalling where the threshold $\frac{n+k}{2}$ comes from, we will need a few more definitions.

\begin{definition}
Given a compact set \(E\subset \R^n\), an \(\alpha\)-Frostman measure \(\mu\) on \(E\) is a nonzero finite Borel measure supported on \(E\) satisfying
\[
\mu(B(x,r))\lesssim r^\alpha,
\qquad \text{for all } x\in \R^n \text{ and } r>0.
\]
By Frostman's lemma \cite[Theorem 2.7]{Mattilabook2015}, such measures exist for every \(0<\alpha<\dim(E)\).
\end{definition}

\begin{definition}\label{def: restrictedmeasure}
    Given a Frostman measure $\mu$ supported on a compact $E\subset \R^n$ and $A$ a Borel set in $\R^n$, denote by $\mu_{A}$ the restriction of $\mu$ to $A$. That is, $\mu_{A}(B)=\mu(B\cap A)$ for any Borel set $B\subset \R^n$. If $A$ is compact set with $\mu(A)>0$, then the new measure $\mu_A$ is a Frostman measure supported on $A$.

\end{definition}


\begin{definition}\label{transversality}
   Let $1\leq k\leq n$. We will say that a collection $\{E_1,\dots, E_k\}$ of compact sets of $\R^n$ is transverse if the following holds:
   \[
 |x_1\wedge\cdots\wedge x_k|\geq c
 \qquad\text{for all }x_i\in E_i.
\]
Here \( |x_1\wedge\cdots\wedge x_k| \) denotes the \(k\)-dimensional volume of the parallelepiped spanned by the vectors \(x_1,\dots,x_k\), viewed as vectors based at the origin. In particular, $x_1,x_2,\dots, x_k$ are linearly independent.
\end{definition}

\begin{definition}[Abundance of \(L^2\) pinned \(k\)-stars]
Let \(E\subset \mathbb R^n\) be compact and let \(\mu\) be a finite Borel
measure supported on \(E\). We say that \(E\) has an
\emph{abundance of \(L^2\) pinned \(k\)-stars relative to \(\mu\)}
if the following holds: 

For every collection of pairwise separated compact sets
\(E_1,\dots,E_{k+1}\subset E\setminus\{0\}\) with \(\mu(E_i)>0\) for each \(i\) and $\{E_1,E_2,\dots ,E_k\}$ transverse, one has $$(d_{x_1,\dots,x_k})_*(\mu_{E_{k+1}})\in L^2(\mathbb R^k)\text{ for }\mu_{E_1}\times\cdots\times\mu_{E_k}\text{ a.e. }(x_1,\dots ,x_{k})\in E_1\times \dots \times E_{k},$$ for \(\mu_{E_i}\) as in Definition \ref{def: restrictedmeasure} and $d_{x_1,x_2,\dots ,x_k}(y):=(|x_1-y|,|x_2-y|,\dots ,|x_k-y|)$. 
\end{definition}

In the definition above, compact sets \(E_1,\dots,E_m\subset\mathbb R^n\) are said to be
\emph{pairwise separated} if, for each \(i\neq j\), there exists \(c_{ij}>0\) such that
\[
|x_i-x_j|\geq c_{ij}
\qquad
\text{for all }x_i\in E_i,\ x_j\in E_j.
\]
Equivalently, since the sets are compact, they are pairwise disjoint with positive mutual
distance.

\begin{remark}
Whenever we say that there is an \emph{abundance of pins}
\((x_1,\dots,x_k)\in E^k\) with respect to \(\mu\) for which a given
property holds, we mean the following: for every collection of pairwise
separated compact sets \(E_1,\dots,E_k\subset E\) with \(\mu(E_i)>0\) and $\{E_1,E_2,\dots ,E_k\}$ transverse, the property holds for \(\mu_{E_1}\times\cdots\times\mu_{E_k}\) almost every \((x_1,\dots,x_k)\in E_1\times\cdots\times E_k\).
\end{remark}

With these definitions in mind, the result for $k$-stars proved in \cite{IPPS22} reads as follows.

\begin{theorem}[\cite{IPPS22}]\label{thm:k starinIPPS}

Assume $n\geq 2$ and $1\leq k<n$. Let $E\subset \mathbb R^n$ be compact with
$\dim(E)>\frac{n+k}{2}$.
Let $\mu$ be an $\alpha$-Frostman measure on $E$, where
$\alpha>\frac{n+k}{2}.$
Then $E$ has an abundance of $L^2$ pinned $k$-stars relative to $\mu$. Consequently, there is an abundance of pins 
$(x_1,\dots,x_k)\in E^k$ such that,
\[
\mathcal L^k\bigl(\Delta^{k\text{-star}}_{x_1,\dots,x_k}(E)\bigr)>0.
\]
\end{theorem}

It was shown in \cite{BFOPR2026} that a \(k\)-star result with the \(L^2\) strength stated above transfers to positive-measure results for \(k\)-admissible pinned graphs. Since improvements for $k$-stars can lead to improved thresholds for Falconer type results for many interesting graphs, the main motivation for this work is to improve the dimensional thresholds in Theorem \ref{thm:k starinIPPS}. Our main theorem is the following improvement, which implies Theorem \ref{thm: simple main}.

\begin{atheorem}\label{thm:main}
     Assume $n\geq 2$ and $1\leq k<n$. Let $E \subset \R^n$ be a compact set with Hausdorff dimension 
     \[\dim(E) > \alpha_{+}(n,k)=\frac{n^2 + nk + k}{2n + 1}=\frac{n + k - 1}{2} + \frac{1}{4} + \frac{2k + 1}{4(2n + 1)}.\]
     
     Let $\mu$ be an $\alpha$-Frostman measure on $E$, where
$\alpha>\alpha_{+}(n,k).$
Then $E$ has an abundance of $L^2$ pinned $k$-stars relative to $\mu$. Consequently, there is an abundance of pins 
$(x_1,\dots,x_k)\in E^k$ with respect to $\mu$ such that,
\[
\mathcal L^k\bigl(\Delta^{k\text{-star}}_{x_1,\dots,x_k}(E)\bigr)>0.
\]
\end{atheorem}

The main analytic step in the proof of Theorem~\ref{thm:main} is a reduction from
pinned $k$-star distance measures to weighted Fourier extension estimates for the
paraboloid. For fixed pins, we rewrite the $L^2$ norm of the density of the associated
pinned $k$-star distance measure in terms of the paraboloid Fourier extension operator evaluated
along the linear span of the corresponding points on the paraboloid. This is the content
of Lemma~\ref{lemma: equalityreplacingliuidentity}, which plays the role that Liu's
identity (see \cite[Theorem 1.9]{LiuL2}) plays in recent works on the Falconer distance problem. 

In the case $k=1$, where $k$-stars reduces to the Falconer distance problem, Fourier restriction estimate has been a widely used tool in the study of the distance problem for decades. It is well known that (see for instance Wolff \cite{Wolff99} and Erdo\u{g}an \cite{Erdogan2006}), through the framework of Mattila integral, one can reduce the original unpinned Falconer distance problem to weighted restriction estimates. In the pinned case, a similar reduction was successfully achieved by Liu \cite{LiuL2}, where he established an $L^2$ identity that directly links the $L^2$ norm of the density of the pinned distance measure to the $L^2$ norm of the Fourier extension of a function supported on the sphere. This identity opened the door to all the most recent improvement towards the pinned version of the Falconer distance conjecture, where the question is reduced to weighted restriction estimates for the sphere. 

The key novelty of our method is that we establish a new $L^2$ identity for all $k$-stars (including the $k=1$ case) that can serve a similar purpose, connecting the $k$-star problem directly to a weighted Fourier restriction problem hence making it possible to introduce analytic tools into the study. Even in the case $k=1$, such a connection to the weighted extension estimates for the paraboloid does not seem to have been observed before and brings new perspectives into the distance set problem. 

Indeed, the new identity we established in the paper is distinct from the classical approaches in the following ways. First, unlike the classical approaches that reduce the problem in $\mathbb{R}^n$ to a weighted restriction problem for the sphere in $\mathbb{R}^n$, our framework reduces it to a weighted restriction problem for the paraboloid in $\mathbb{R}^{n+1}$. It is known that the paraboloid and the sphere display different behaviors in the weighted restriction theory (see e.g. \cite{DuinJLMS}). Hence this new connection may bring additional insights into the theory of distance problems. Second, while the classical approaches reduce the distance problem to the Fourier extension of the Fourier transform of the Frostman measure, our framework instead turns it into the Fourier extension of the Frostman measure itself. In addition, this framework can also be more easily extended to the general $k$-star case, whereas it's unclear to us whether the approach via Liu's identity could be used to study the $k$-stars. More detailed discussions on these comparisons can be found in Section \ref{ref:ProofofThmA}.

Another interesting feature of our proof is the role played by the averaging operator over a lower dimensional sphere in $\mathbb{R}^n$. It is well known that, in the $k=1$ case, the pinned distance measure is naturally related to the spherical averaging operator where the sphere $S^{n-1}$ is centered at the pin. For $k\geq 2$, a similar averaging operator is brought into the picture where the sphere is of higher codimension and is jointly determined by all the $k$ pins. Such an operator seems to be intrinsically more difficult to understand than the standard spherical average, and we expect that improved understanding of it would shed new light on the $k$-star problem.

As mentioned earlier, the role of Theorem \ref{thm:main} is broader than the $k$-star configuration itself. Since it gives an improved abundance theorem for $L^2$ pinned $k$-stars, it can be used
as the input in the graph-building framework of \cite{BFOPR2026}. Consequently,
Theorem~\ref{thm:main} yields the following corollary for $k$-admissible pinned graphs.

\begin{corollary}\label{cor: pinnedgraphs}
    Let $G$ be a graph with vertex set $\mathcal{V}$ and edge set $\mathcal{E}$. Let $\mathcal{P}\subset \mathcal{V}$ be such that no pair of vertices in $\mathcal{P}$ is connected by an edge in $\mathcal{E}$ and denote $m=|\mathcal{P}|$. Assume that $(G,\mathcal{P})$ is a $k$-admissible pinned graph. Then, for $n\geq2$, $1\leq k<n$ and any compact set $E\subset \mathbb{R}^n$ with Hausdorff dimension larger than  $$\alpha_{+}(n,k)=\frac{n + k - 1}{2} + \frac{1}{4} + \frac{2k + 1}{4(2n + 1)},$$ there are distinct points $x_1,\cdots, x_m\in E$, such that the pinned distance set $\Delta^{G}_{\mathcal{P}, x_1,\cdots,x_m}(E)$ has positive $|\mathcal{E}|$-dimensional Lebesgue measure.
\end{corollary}

Thus, the weighted restriction estimates used to prove Theorem~\ref{thm:main} translate, through the graph-building framework of \cite{BFOPR2026}, into improved Falconer-type thresholds for every \(k\)-admissible pinned graph. Several special classes of graphs are of particular interest. In Subsection~\ref{subsection:cyclesandsimplices}, we state several explicit corollaries about these graphs of our theorem. Specifically, we discuss cycles, also known as necklaces, and simplices, both of which have played a central role in previous work on Falconer-type configuration problems.

Theorem \ref{thm:main} also improves several of the area-vector thresholds obtained
in \cite{BFPOR2026II}. The Jacobian method developed there transfers
positive-measure results for pinned distance graphs to the corresponding
vectors of triangle areas whenever the associated length-to-area map has
full rank. Consequently, replacing the previous $k$-star threshold
$\frac{n+k}{2}$ by $\alpha_+(n,k)$ immediately improves the dimensional
hypotheses for the area configurations treated there that are built from
$k$-admissible graphs. These include, for example, pinned fans, fish
graphs, wheels, and circumscribed flowers, in the ranges where $k<n$.

\subsection{Nonempty interior results and dimension estimates for \texorpdfstring{$k$}{k}-stars}

In addition to positive Lebesgue measure type results, our approach to Theorem \ref{thm:main} has the advantage that it can be readily adapted to study other related problems for the $k$-star sets and beyond. Surprisingly, even in the $k=1$ case (the Falconer distance set problem), it yields several new improvements of previously best known results. 

For sets $E\subset \mathbb{R}^n$, rather than seeking dimensional thresholds that guarantee $\Delta(E)$ or $\Delta_x(E)$ have positive measure, one may ask for sufficient conditions ensuring the stronger conclusion that $\Delta(E)$ or $\Delta_x(E)$ has nonempty interior. A classical result of Mattila and Sj\"olin \cite{MS99} states that if $E\subset \mathbb{R}^n$ satisfies $\dim(E)>\frac{n+1}{2}$, then $\operatorname{int}(\Delta(E))\neq\emptyset$. For the pinned analogue, Peres and Schlag \cite{PS00} showed that when $n\geq 3$ and $\dim(E)>\frac{n+2}{2}$, there exists $x\in E$ such that $\operatorname{int}(\Delta_x(E))\neq\emptyset$. Later, Iosevich and Liu \cite{IL19} derived improved exceptional set estimate for the pinned nonempty interior problem in certain regimes and introduced a new method based on local smoothing estimates. More recently, these thresholds were improved and extended by the first two authors of this paper, Foster and Palsson in \cite{BFOP2026}. In particular, they proved that in the plane the same conclusion holds under the condition $\dim(E)>7/4$, and that in dimension $n=3$ the threshold can be lowered to $12/5=2.4$. They also showed that, if the local smoothing conjecture for the wave equation were established for $n\geq 3$, then the threshold could be further reduced to $\frac{n+2}{2}-\frac{1}{2n}$, representing an improvement of $\frac{1}{2n}$ over the Peres--Schlag threshold.

\medskip
The nonempty interior problem extends naturally to $k$-star distance sets. When $k=1$, this reduces precisely to the pinned nonempty interior problem for distance sets discussed above.

\begin{question}
Let $k\geq 1$ and $E\subset \mathbb{R}^n,\,n\geq 2$. Under what hypotheses on the Hausdorff dimension of $E$ can one guarantee the existence of points $x_1,\dots,x_k\in E$ such that
\[
\Delta_{x_1,\dots,x_k}^{k\text{-star}}(E)
:=
\{(|x_1-x|,|x_2-x|,\dots,|x_k-x|):x\in E\}
\]
has nonempty interior in $\mathbb{R}^k$?
\end{question}

To the best of our knowledge, apart from the case $k=1$, this question has not been explicitly studied before. By adapting ideas from the proof of Theorem~\ref{thm:main}, we obtain the first nonempty-interior thresholds for pinned $k$-star distance sets with $k\geq 2$.

\begin{atheorem}\label{thm: nonemptyinteriorforkstars}
    Let $n\geq 3$ and \(1\leq k<\frac{n}{2}\). Let \(E\subset \mathbb R^n\) be a compact set satisfying
\[
\dim(E)>\alpha^{\circ}(n,k)
:=
\frac{n+2k-1}{2}+\frac{1}{4}+\frac{4k+1}{4(2n+1)}.
\]
Let \(\mu\) be an \(\alpha\)-Frostman measure on \(E\), where
\(\alpha>\alpha^{\circ}(n,k)\). Then there is an abundance of pins
\((x_1,\dots,x_k)\in E^k\) with respect to \(\mu\) such that
$\Delta_{x_1,\dots,x_k}^{k\text{-star}}(E)$
has nonempty interior in \(\R^k\).
\end{atheorem}

The condition \(n>2k\) in the statement above is imposed to ensure that the dimensional hypothesis is nontrivial, that is, $\alpha^{\circ}(n,k)<n$. So far, there isn't a mechanism in the literature that allows us to transfer thresholds for nonempty interior of $k$-stars to nonempty interior of $k$-admissible graphs. In some cases, one can succeed in using the nonempty interior result for $k$-stars to obtain nonempty interior for more complicated distance graphs. We will give an example in the next subsection by providing an application of the nonempty interior result for $2$-stars to the nonempty interior of even necklaces.

For the case $k=1$, which is of independent interest, we can further improve the threshold above by adapting ideas from \cite{DZ2019}. In that case, we obtain the following.

\begin{atheorem}\label{thm: nonemptyinteriorfor1star}
Let $n\geq 3$, and let $E\subset \mathbb{R}^n$ be a compact set satisfying
\[
\dim(E)>\frac{n}{2}+\frac{3}{4}+\frac{3}{4(2n-1)}= \frac{n(n + 1)}{ 2n - 1}.
\]
Then there exists $x\in E$ such that $\Delta_x(E)$ contains a nontrivial interval.
\end{atheorem}

\begin{remark}
   In the case $k=1$, let us compare the thresholds in Theorem \ref{thm: nonemptyinteriorforkstars} and \ref{thm: nonemptyinteriorfor1star} with the previously best known ones in \cite{BFOP2026}. For all $n\geq 5$, the threshold in Theorem \ref{thm: nonemptyinteriorforkstars} is smaller than
$\frac n2+1-\frac{1}{2n}$,
which is the threshold that would follow from the argument in \cite{BFOP2026} and the sharp local smoothing conjecture for the wave equation. In all dimensions, the threshold in Theorem \ref{thm: nonemptyinteriorfor1star} is always smaller than that in Theorem \ref{thm: nonemptyinteriorforkstars}. For every \(n\geq 4\), the threshold in Theorem~\ref{thm: nonemptyinteriorfor1star} is smaller than
what would follow from the argument in \cite{BFOP2026} and the sharp local smoothing conjecture. In $n=3$, Theorem \ref{thm: nonemptyinteriorfor1star} recovers the same threshold $12/5=2.4$ obtained in \cite{BFOP2026}. For $n=2$, we obtain no information, so the threshold $7/4$ in \cite{BFOP2026} remains the best. 
\end{remark} 

When the dimension of the set $E$ is smaller than $\alpha_+(n,k)$, Theorem \ref{thm:main} does not apply. However, our method can still be adapted to derive lower bounds for the Hausdorff dimension of the $k$-star sets of $E$ in this regime. The exact statement of our result is given in Theorem \ref{thm:kstar_dimension_lower} in Section \ref{sec:nonemptyanddimresults}. In the $k=1$ case, such dimensional estimates for the distance set has attracted a great amount of attention in recent years and exciting new progress has been obtained via many different approaches ranging from Fourier analysis \cite{Liu2020}, projection theory \cite{Shmerkinintheplane}, to computability theory \cite{fiedlerstull23}. We refer the interested reader to these works and the references therein for a detailed recount of the recent progress and state-of-the-arts results in different regimes. In the case $k\geq 2$, Theorem \ref{thm:kstar_dimension_lower} seems to be the first of its kind. Just like the Lebesgue measure results, such dimension estimates for $k$-stars also play a fundamental role in the theory of general distance graphs and can be used to derive similar results for $k$-admissible graphs. For instance, as an application of the $k=2$ case of Theorem \ref{thm:kstar_dimension_lower}, we obtain in Corollary \ref{cor: dim triangle}, for the first time, a dimension estimate for pinned triangle sets. The proof is inspired from the work of the second listed author and Taylor \cite{OT2020}, where a mechanism translating dimension estimates for pinned distances to chains and trees was established. However, the method in \cite{OT2020} was not strong enough to handle cycles.

\subsection{Applications of Theorem \ref{thm:main} and Theorem \ref{thm: nonemptyinteriorforkstars} to cycles and $k$-simplices}\label{subsection:cyclesandsimplices}


Consider the particular case in which the graph of interest is a cycle
$C_l$ with $l$ vertices. For even cycles, Greenleaf, Iosevich, and
Pramanik~\cite{GIP17} studied the constant-gap problem. More precisely,
they proved that, if $d\geq 4$ and
\[
    \dim(E)>\frac{d+3}{2},
\]
then the set of common edge lengths
\(\bigl\{
        t>0 : (t,\ldots,t)\in \Delta^{C_l}(E)
    \bigr\}
\)
contains a nontrivial interval. More recently, Iosevich, Magyar,
McDonald, and McDonald~\cite{IMMM25} obtained an analogous result for
$4$-cycles in $\mathbb R^3$ under the assumption
\[
    \dim(E)>\frac{53-\sqrt{337}}{12}.
\]
These results concern the constant-gap slice of the full cycle distance
set. By contrast, our positive-measure result concerns the full set of
edge-length vectors in $\mathbb R^l$.

Since cycles with no adjacent pinned vertices are $2$-admissible,
Corollary~\ref{cor: pinnedgraphs} implies the following result, which
improves the threshold $\frac{n+2}{2}$ originally established in
\cite{BFOPR2026}.
\begin{corollary}[Improvement for $l$-cycles]\label{cor: necklaces}

 Let $l\geq 3$. Let $C_l$ be a $l$-cycle graph with $m\geq0$ pins listed in $\mathcal{P}$, such that no pair of vertices in $\mathcal{P}$ shares an edge. Then, if $n\geq 3$ and $E\subset \R^n$ is a compact set satisfying that $$\dim(E)>\frac{n+1}{2}+\frac{1}{4}+\frac{5}{4(2n+1)},$$ 
 then the corresponding pinned $l$-cycle distance set has positive $l$-dimensional
Lebesgue measure, in the sense that
 
$$\mathcal{L}^l\left(\Delta_{\mathcal{P},x_1,x_2,\dots ,x_m}^{C_l}(E))\right)>0$$
for some  $x_1,x_2,\dots ,x_m\in E$. 
\end{corollary}

Another important graph is $K_{k+1}$, the complete graph in $(k+1)$-vertices, also known as a $k$-dimensional simplex. In that case $\mathcal{V}=\{v_1,\dots,v_{k+1}\}$ and $\mathcal{E}=\{(i,j)\colon 1\leq i<j\leq k+1\}$, with associated unpinned distance set 
$$\Delta^{k-simplex}(E):=\{(|x_i-x_j|)_{1\leq i<j\leq k+1}\colon x_i\in E\text{ for all }1\leq i\leq k+1\},$$
and pinned counterpart 
$$\Delta_{x_1}^{k-simplex}(E):=\{(|x_i-x_j|)_{1\leq i<j\leq k+1}\colon x_i\in E\text{ for all }2\leq i\leq k+1\}.$$

 In the case $k=2$, $K_3$ is a triangle. The pinned distance set (with pin in $x_1$) in that case looks like
$$\Delta^{triangle}_{x_1}(E)=\{(|x_1-x_2|,|x_2-x_3|,|x_3-x_1|)\colon x_2,x_3\in E\textnormal{ distinct}\},$$
and its unpinned counterpart is 
$$\Delta^{triangle}(E)=\{(|x_1-x_2|,|x_2-x_3|,|x_3-x_1|)\colon x_1,x_2,x_3\in E\textnormal{ distinct}\}.$$

The study of triangle configurations realized in thin sets of $\R^n$ has a rich literature, and it goes back to Greenleaf and Iosevich \cite{GI12} who proved that
$\mathcal{L}^3(\Delta^{triangle}(E))>0$
whenever $E\subset \R^2$ with $\dim(E)>7/4$. That was improved to the threshold $8/5$ by Greenleaf, Iosevich, Liu and Palsson in \cite{GILP15}. The group action method in that paper takes spherical-average decay inputs into dimensional thresholds for unpinned results for positive measure of $k$-simplices, and it can be stated as follows.

\begin{proposition}[Group action criterion given spherical-average decay input \cite{GILP15}]\label{prop:GILP_branch}
Let $n\geq 2$. Assume that for every $\alpha$-Frostman measure $\mu$ on $\R^n$ one has
\begin{equation}\label{eq:beta_def}
  \int_{S^{n-1}} |\widehat\mu(R\omega)|^2\,d\sigma(\omega)
  \lesssim_{\mu,\epsilon} R^{-\beta_n(\alpha)+\epsilon},
  \qquad R\geq 1.
\end{equation}
for some $\beta_n(\alpha)>0$. Let $1\leq k\leq n $. Then, if $E\subset \R^n$ is compact and
\begin{equation}\label{eq:GILP_condition}
  \alpha+\frac{1}{k}\beta_n(\alpha)>n
\end{equation}
for some $\alpha<\dim(E)$, then
\[
  \mathcal{L}^{\binom{k+1}{2}}\bigl(\Delta^{k-simplex}(E)\bigr)>0.
\]
\end{proposition}

Since \cite{GILP15} was published, improved spherical-average decay estimates have been proved in $n\geq 3$ by Du and Zhang \cite{DZ2019}, who obtained $\beta_{n}(\alpha)=\frac{n-1}{n}\alpha$. Incorporating this improvement, the group action method from \cite{GILP15} leads to the following 

\begin{theorem}[\cite{GILP15},\cite{DZ2019}]
    Assume $n\geq 2$ and $1\leq k\leq n$. Let $E\subset \R^n$ be a compact set with $\dim(E)>8/5$ if $n=k=2$ or $$\dim(E)>\alpha_{\textnormal{DZ/GILP}}(n,k):=\frac{kn^2}{(k+1)n-1}$$ if $n\geq 3$, then $\mathcal{L}^{\binom{k+1}{2}}(\Delta^{k-simplex}(E))>0.$
\end{theorem}

For the pinned counterpart, the following result was obtained in \cite{IPPS22}, which, to our best knowledge, is the best pinned $k$-dimensional simplex result for $k\geq 2$ in the literature.

\begin{theorem}[\cite{IPPS22}]
    Assume $n\geq 2$ and $1\leq k<n$. Let $E\subset \R^n$ be a compact with $\dim(E)>\frac{n+k}{2}$. Then there exists $x\in E$ such that 
    $$\mathcal{L}^{\binom{k+1}{2}}(\Delta^{k-simplex}_{x}(E))>0.$$
\end{theorem}

Since the complete graph on \(k+1\) vertices, corresponding to a \(k\)-simplex, is \(k\)-admissible even when one vertex is pinned, we can apply Corollary \ref{cor: pinnedgraphs} to obtain the following improvement over \cite{IPPS22}.

\begin{corollary}[New threshold for $k$-simplices] \label{cor: simplices}
Let $1\leq k<n$. If $E\subset \R^n$ is a compact set satisfying that \[\dim(E) > \alpha_{+}(n,k):=\frac{n + k - 1}{2} + \frac{1}{4} + \frac{2k + 1}{4(2n + 1)},\]
then there exists $x\in E$ such that one has positive measure of $k$-simplices pinned at $x$, namely  
$$\mathcal{L}^{\binom{k+1}{2}}\left(\Delta_{x}^{k\text{-simplex}}(E)\right)>0.$$ 
\end{corollary}

\begin{remark}[Comparison with the unpinned simplex thresholds]
Since
\[
\Delta_x^{k\text{-simplex}}(E)\subset \Delta^{k\text{-simplex}}(E)
\]
whenever \(x\in E\), Corollary~\ref{cor: simplices} also implies an unpinned
positive-measure result for \(k\)-simplices under the condition
\[
\dim(E)>\alpha_{+}(n,k).
\]
It is therefore natural to compare this threshold with the one obtained by combining
the group-action method of~\cite{GILP15} with the Du--Zhang spherical-average decay
estimate~\cite{DZ2019}, namely
\[
\alpha_{\mathrm{DZ/GILP}}(n,k)
:=
\frac{kn^2}{(k+1)n-1}.
\]

For \(n=2\) and \(k=2\), Corollary~\ref{cor: simplices} does not apply, and the
relevant unpinned triangle threshold is the GILP threshold \(8/5\). For \(n=3\), the
GILP/DZ threshold is better for both \(k=2\) and \(k=3\): in particular,
\[
\alpha_{\mathrm{DZ/GILP}}(3,2)=\frac94
<
\frac{17}{7}
=
\alpha_{+}(3,2),
\]
while \(k=3\) lies outside the nontrivial range of Corollary~\ref{cor: simplices}.

In dimensions \(n\geq 4\), there is a split. For \(2\leq k\leq n-2\), one has
\[
\alpha_{+}(n,k)<\alpha_{\mathrm{DZ/GILP}}(n,k),
\]
so Corollary~\ref{cor: simplices} gives the better currently available threshold for
positive measure of \(k\)-simplices. On the other hand, for \(k=n-1\) and \(k=n\), the
GILP/DZ threshold is better. In the case \(k=n\), our pinned \(k\)-star theorem is
vacuous, since the nontrivial range for pinned \(k\)-stars requires \(k<n\).

\end{remark}

As a corollary of our new threshold for $2$-stars, we improve the best known threshold for pinned triangles in all dimensions $n\geq 3$. We state this explicitly as a corollary.

\begin{corollary}\label{cor: triangles} (Improvement for pinned triangles) 
    Let $n\geq 3$. For any $E\subset \R^n$ with $\dim(E)>\frac{n+1}{2}+\frac{1}{4}+\frac{5}{4(2n+1)}$, there exists $x\in E$ such that $\mathcal{L}^3(\Delta_x^{triangle}(E))>0$.
    
\end{corollary}
\begin{remark}
The preceding corollary improves the pinned triangle threshold, but it does not close
the gap between the best pinned and unpinned results in low dimensions. In dimension
\(n=2\), the unpinned triangle problem has the nontrivial threshold \(8/5\), while no
nontrivial pinned analogue is currently known. In dimension \(n=3\), our result lowers
the pinned threshold from \(5/2\) to \(17/7\), but this remains above the unpinned
threshold \(\alpha_{DZ/GILP}(3,2)=9/4\) obtained by group-action methods.
\end{remark}

\medskip

Moving on to applications of Theorem \ref{thm: nonemptyinteriorforkstars}, although the nonempty-interior result for pinned \(k\)-stars does not currently feed into a general graph-building theorem for nonempty interior for distance sets of \(k\)-admissible graphs, it can still be used as a building block in some specific configurations. We illustrate this for even cycles pinned at every other vertex, using the nonempty-interior result for \(2\)-stars.
\begin{corollary}\label{cor:interiorforevennecklaces}
Assume $n\geq 5$, and let $E\subset \R^n$ with
\[
\dim(E)>\alpha^{\circ}(n,2)
=\frac{n+3}{2}+\frac{1}{4}+\frac{9}{4(2n+1)}.
\]
For every $l\geq 2$, let $C_{2l}$ be an even cycle with $2l$ edges, and consider the corresponding pinned distance set with pins at every other vertex, namely
\[
\Delta_{x_1,x_3,\dots,x_{2l-1}}^{C_{2l},\text{l pins}}(E)
:=
\bigl\{(|x_1-x_2|,
|x_2-x_3|,|x_3-x_4|,\dots,
|x_{2l-1}-x_{2l}|,|x_{2l}-x_1|)
:\, x_{2i}\in E,\ \forall\,1\leq i\leq l\bigr\}.
\]
Then there exists
$x_1,x_3,\dots,x_{2l-1}\in E$
such that
\[
\Delta_{x_1,x_3,\dots,x_{2l-1}}^{C_{2l},\text{l pins}}(E)
\]
has nonempty interior in $\R^{2l}$.
\end{corollary}

\begin{remark}
Nonempty interior for odd cycles is still out of reach with our methods. For example, a triangle is a $2$-admissible graph, but it is not clear whether a nonempty-interior result for $2$-stars necessarily implies
a nonempty-interior result for triangle distance sets. We refer to \cite{PR2023} for an unpinned nonempty interior result for triangles with threshold $\frac{2n}{3}+1$ in $\R^n,\,n\geq 4$. 
\end{remark}

\subsection{Notation} 

\begin{enumerate}
    \item $\mathbb{P}^{n}$ will denote the truncated paraboloid in $\R^{n+1}$ defined by
    \[\mathbb{P}^{n} = \{(x, |x|^2) : x \in \R^n, |x| \leq 1\}.\] 
    We might also consider the infinite paraboloid 
    $$\mathbb{P}^n_{\infty}:=\{(x, |x|^2) : x \in \R^n\}.$$
    \item For points $y \in \R^{n+1}$, we will use the notation $y = (x, t)$ where $x \in \R^n$ and $t \in \R$.
    \item Even if not explicitly stated, we always assume $n\geq 2$.
    \item For $f: \R^n \to \C$, $\mc{E}f$ will denote the Fourier extension operator for the truncated paraboloid
    \[\mc{E}f(y) = \int_{B^{n}(0,1)}f(z)e^{2\pi i(x \cdot z + t|z|^2)} \,dz,\,\text{ for all }y=(x,t)\in \R^{n+1}. \]
    \item We will use $E, E_j$ to denote subsets of Euclidean space $\R^n$ and $F,F_j$ to denote subsets of the paraboloid $\mathbb{P}^{n} \subset \R^{n+1}$.
    \item  $\pi: \mathbb{P}^n_{\infty} \setminus\{0\}\to \R^n\setminus\{0\}$ will denote the map
    \[ \pi(y) = \frac{-\ov{y}}{2|\ov{y}|^2}.\] One can check that the inverse of $\pi$ is given by $\pi^{-1}(x)=\left(\frac{-x}{2|x|^2},\frac{1}{4|x|^2}\right).$
    \item $x_1,...,x_k$ will denote points in $\R^n\setminus\{0\}$ and $y_1, ..., y_k$ will denote the corresponding points
    \[y_j = \pi^{-1}(x_j)\in \mathbb{P}^n_{\infty}.\]
    \item $\bm{x}$ will denote the vector $(x_1, ..., x_k) \in \R^{nk}$ and similarly for $\bm{y}$ ($k$ will be fixed).
    \item For all $k\geq 1$, we denote the distance maps $$d_{x_1, x_2, \dots ,x_k}(x_{k+1}) := (|x_1 - x_{k+1}|,|x_2-x_{k+1}|,\dots ,|x_k-x_{k+1}|)$$ and the dot product maps $$p_{y_1, \dots, y_k}(y_{k+1}) := (y_1 \cdot y_{k+1}, y_2\cdot y_{k+1}, \dots, y_k\cdot y_{k+1}).$$
\end{enumerate}

\subsection*{Organization of the paper} The paper is organized as follows. In Section~\ref{sec:reductiontorestriction}, we reduce the \(k\)-star problem in $\mathbb{R}^n$ to a weighted restriction estimate for the paraboloid in \(\R^{n+1}\). In Section~\ref{sec:facts_weight}, we study the geometric and dimensional properties of the weights arising in this reduction. The proof of Theorem \ref{thm:main} is presented in Section \ref{ref:ProofofThmA}. In Section~\ref{sec:nonemptyanddimresults}, we prove Theorem~\ref{thm: nonemptyinteriorforkstars}, the nonempty-interior counterpart of Theorem~\ref{thm:main}, establish dimension estimates below the dimensional threshold in Theorem~\ref{thm:main}, and illustrate how dimension estimates for \(2\)-stars can be used to obtain dimension estimates for triangles. In the appendix, we adapt ideas from Du--Zhang \cite{DZ2019} to prove Theorem~\ref{thm: nonemptyinteriorfor1star}, and we also give an application of Theorem~\ref{thm:main} to derive a discrete variant of it. To our knowledge, this seems to be the first result concerning pinned $k$-stars in the discrete setting.

\subsection*{Acknowledgement}
Y.O. is supported in part by NSF DMS-2142221. The authors would like to thank Jill Pipher for many helpful conversations on the topic of this paper.

\section{Reduction to Weighted estimates of the Paraboloid Extension Operator}\label{sec:reductiontorestriction}

In this section, we relate the key $L^2$ integral of the density of a pinned $k$-star to a weighted Fourier extension estimate for the paraboloid. Throughout the section, we assume that $n\geq 2$ and $1\leq k<n$. We first introduce some definitions. 
\begin{definition}
    Given points $\bm{x} = (x_1, ..., x_k) \in (\R^n)^{k}$, we let $A_{\bm{x}}$ denote the affine span of the points $x_1, ..., x_k$. In other words, $A_{\bm{x}}$ is the smallest affine subspace of $\R^n$ containing the points $x_1, ..., x_k$, or equivalently,
    \[
A_{\mathbf x}
=
\left\{
\sum_{j=1}^k \lambda_j x_j:
\lambda_1,\dots,\lambda_k\in\mathbb R,\ 
\sum_{j=1}^k \lambda_j=1
\right\}.
\]
\end{definition}

In the same way that the distance map $d_x = |\cdot - x|$ is closely related to the spherical average operator, the push-forward of $d_{\bm{x}}$ is related to a lower-dimensional spherical averaging operator. To define this operator, we make the following definitions.

\begin{definition}
    Let $\textup{Aff} \subset (\R^n)^k$ be the set of points that are affinely linearly independent. That is,
    \[\textup{Aff} := \{\bm{x} = (x_1, ..., x_k) \in (\R^n)^k : \text{dim}(A_{\bm{x}}) = k - 1\}.\]
\end{definition}

\begin{definition}
    For $\bm{x} = (x_1, \ldots, x_k) \in \textup{Aff}$ and $\bm{t} = (t_1, \ldots, t_k) \in (0,\infty)^k$, we denote 
    \[S^{n - k}(\bm{x}, \bm{t}) := S^{n - 1}(x_1, t_1) \cap \cdots \cap S^{n - 1}(x_k, t_k)
    .\]
    In other words,
    \[S^{n-k}(\bm{x},\bm{t}) = d_{\bm{x}}^{-1}(\bm{t}) =\{x_{k+1}\in\R^n: |x_j-x_{k+1}|=t_j,\ 1\leq j\leq k\}.\]
\end{definition}

It's a well-known fact that under the assumption $\bm{x} \in \text{Aff}$, $S^{n-k}(\bm{x}, \bm{t})$ is either empty, a point, or a $(n - k)$-dimensional sphere (see \cite{MLG17} Theorem 2.1). Let $D_{\bm{x}}: \R^n \to \R$ denote the determinant of the Jacobian of the distance map $d_{\bm{x}}:\R^n\rightarrow \R^k$:
\[D_{\bm{x}}(x_{k+1}) = (\det(\nabla d_{\bm{x}}(x_{k+1})\nabla d_{\bm{x}}(x_{k+1})^T))^{1/2}.\] 
  Note that $\nabla d_{\bm{x}}(x_{k+1})$ is the $k \times n$ matrix given by 
\[\left(\frac{x_1 - x_{k+1}}{|x_1 - x_{k+1|}}, \dots , \frac{x_k - x_{k+1}}{|x_k - x_{k+1}|}\right)^T.\]

Therefore
\[D_{\bm{x}}(x_{k+1}) > 0 \quad \Leftrightarrow \quad \text{rank}(\nabla d_{\bm{x}}(x_{k+1})) = k \quad \Leftrightarrow \quad x_{k+1} \not \in A_{\bm{x}}.\]
For the first equivalence, we used the fact that $\text{rank}(AA^{T}) = \text{rank}(A)$ for any matrix $A$, and for the second equivalence, we used that $x_{k+1} \not \in A_{\bm{x}}$ means $\{x_j - x_{k+1}\}_{1 \leq j \leq k}$ is linearly independent. It follows that the set of critical points of $d_{\bm{x}}$ has dimension $k - 1$:
\[\text{dim}\{x_{k+1} \in \R^n : D_{\bm{x}}(x_{k+1}) = 0 \} = \text{dim}(A_{\bm{x}}) = k - 1\]
and thus the set of critical values 
\[C_{\bm{x}} := d_{\bm{x}}(A_{\bm{x}} \setminus \{x_1, ..., x_k\})\]
has Lebesgue measure zero in $(0, \infty)^k$. Therefore by the regular value theorem, $d_{\bm{x}}^{-1}(\bm{t}) = S^{n - k}(\bm{x}, \bm{t})$ is an $(n-k)$-dimensional sphere (or empty) for all $\bm{t} \not \in C_{\bm{x}}$. 

Therefore 
\[C_{\bm{x}} = \{\bm{t} \in \R^k : S^{n-k}(\bm{x}, \bm{t}) \text{ is a point}\}\]

Also let 
\[Z_{\bm{x}} = \{\bm{t} \in \R^k : S^{n-k}(\bm{x}, \bm{t}) = \emptyset\}.\] 
We are now ready to define the spherical averaging operator.

\begin{definition}\label{def:spherical_average} Let $\bm{x} \in \textup{Aff}$. We define
    \[\mc{A}_{\bm{t}}f(\bm{x}) := 
        \int_{S^{n-k}(\bm{x}, \bm{t})}f(y) \, d\sigma_{\bm{x},\bm{t}}(y), \quad \bm{t} \not \in C_{\bm{x}} \cup Z_{\bm{x}}\]
and define $\mc{A}_{\bm{t}}f(\bm{x}) = 0$ for $\bm{t} \in C_{\bm{x}} \cup Z_{\bm{x}}$. Here the surface measure $\sigma_{\bm{x},\bm{t}}$ on $S^{n - k}(\bm{x}, \bm{t})$ is
$$d\sigma_{\bm{x},\bm{t}}(x_{k+1}):=\frac{d\mathcal{H}^{n-k}|_{S^{n-k}(\bm{x},\bm{t})}(x_{k+1})}{D_{\bm{x}}(x_{k+1})}$$
where $\mc{H}^{n-k}|_{S^{n-k}}$ is the \emph{unnormalized} $(n-k)$-dimensional Hausdorff measure on $S^{n - k}(\bm{x}, \bm{t})$.
\end{definition}

\begin{remark}
    If $\bm{t} \not \in C_{\bm{x}}$ then by definition, $\bm{t}$ is a regular value of $d_{\bm{x}}$ and so for $x_{k+1} \in S^{n-k}(\bm{x}, \bm{t})$, we know that $D_{\bm{x}}(x_{k+1}) > 0$. Therefore the definition of $\sigma_{\bm{x}, \bm{t}}$ makes sense. 
\end{remark}

\begin{remark}
The reason we choose this definition is so that we can identify the measure $\mc{A}_{\bm{t}}f(\bm{x}) \, d\bm{t}$ with the pushforward of $f$ under the distance map $d_{\bm{x}}$. Indeed, using the coarea formula, we can compute 
\begin{align*}
\int_{\R^k}g(\bm{t})d((d_{\bm{x}})_{*}(f \, dx_{k+1}))(\bm{t}) & = \int_{\R^n}g(|x_{k+1} - x_1|, ..., |x_{k+1} - x_k|)f(x_{k+1}) \, dx_{k+1} \\
& = \int_{(0,\infty)^k}\int_{S^{n-k}(\bm{x}, \bm{t})}g(\bm{t})f(x_{k+1})D_{\bm{x}}(x_{k+1})^{-1} \, d\mc{H}^{n-k}(x_{k+1}) \, d\bm{t} \\
& = \int_{(0,\infty)^k}g(\bm{t})\mc{A}_{\bm{t}}f(\bm{x}) \, d\bm{t}. 
\end{align*}
In particular, by setting $g = 1$, we obtain the generalized polar coordinates formula
\[\int_{\R^n}f(x_{k+1}) \, dx_{k+1} = \int_{(0,\infty)^k}\int_{S^{n-k}(\bm{x}, \bm{t})}f(x_{k+1}) \, d\sigma_{\bm{x},\bm{t}}(x_{k+1}) \, d\bm{t}.\]
Also note that the pushforward measure $(d_{\bm{x}})_{*}(f)(\bm{t})$ and $\mathcal{A}_{\bm{t}}(f)(\bm{x})$ have the same support as measures in the variable $\bm{t}$, contained in the $k$-star set of interest.
\end{remark}

\begin{remark}
    By setting $k = 1$, the definition of $\mc{A}$ gives the spherical average operator with respect to \emph{unnormalized} surface measure:
    \[\mc{A}_tf(x) = \int_{S^{n-1}(x,t)}f(y) \, d\sigma_{x,t}(y).\]
    In the literature, the spherical average operator is usually defined with the normalized surface measure, but we choose to define our operator to be unnormalized because it makes the coarea identity above more direct.
\end{remark}

 Next, for $x \in \R^n\setminus\{0\}$, let $\Gamma_{x}: (-1/4, \infty) \to (0, \infty)$ denote the map 
$$\Gamma_x(s) = |x|\sqrt{1 + 4s}.$$

The importance of this map in our context comes from the fact observed in \cite{IPPS22} that for a fixed point $y\in \mathbb{P}^n_{\infty}\setminus\{0\}$, and $s>-1/4$, then for any $x\in \R^n$
\begin{equation}\label{distancedotproductrelation}
   y\cdot (x, |x|^2) = s  \quad \Leftrightarrow  |\pi(y) - x| = \Gamma_{\pi(y)}(s).
    \end{equation}

Checking the equivalence above is straightforward by squaring and unraveling definitions in the equality on the right-hand side. This equivalence carries rich geometric information. Indeed, for $y\in \mathbb{P}^n_{\infty}$ consider the following level set of the map that takes the dot product between a point in the paraboloid and $y$, 
$$\mathcal{L}_y(s):=\{(x,|x|^2)\in \mathbb{P}^n_{\infty}\colon y\cdot (x,|x|^2)=s\},$$
which is the intersection of the paraboloid $\mathbb{P}^n_{\infty}$ with a hyperplane of normal direction $y$. Identity (\ref{distancedotproductrelation}) tell us if we orthogonally project $\mathcal{L}_y(s)$ onto the first $n$ coordinates, we obtain a sphere in $\R^n$ centered at $\pi(y)=\frac{-\bar{y}}{2|\bar{y}|^2}$ and radius $\Gamma_{\pi(y)}(s):=|\pi(y)|\sqrt{1 + 4s}$, that is, a level set of the map $d_{\pi(y)}(x)=|x-\pi(y)|$. We illustrate that in Figure \ref{fig:paraboloid-planes} with $n=2$ and $y=(0,-1,1)$, in which case $\pi(y)=(0,1/2)$.

\begin{figure}[h]
    \centering
\includegraphics[width=0.65\textwidth]{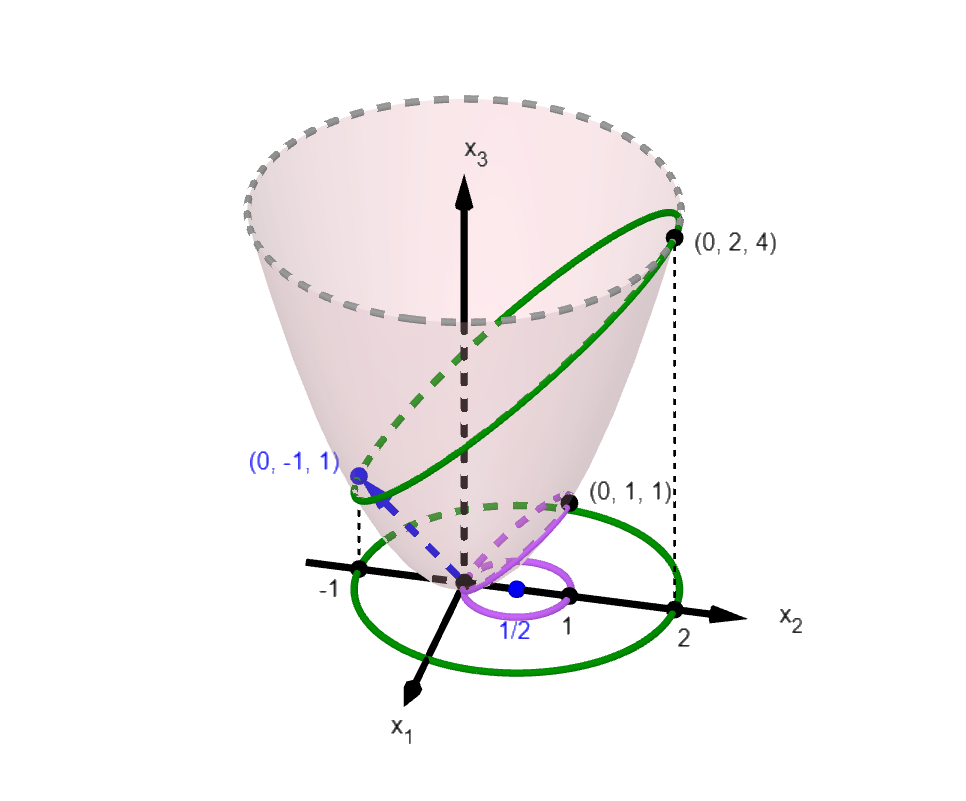}
    \caption{Representation of $\mathcal{L}_y(s)$ for $y=(0,-1,1)$ and $s=0$ and $s=2$, and the corresponding projected circles in the $(x_1,x_2)$ plane centered at $(0,1/2)$. For $s=-1/4$, $\mathcal{L}_y(-1/4)$ would be the single point $(0,1/2,1/4)$.}
    \label{fig:paraboloid-planes}
\end{figure}

Next, we will introduce some notation that will be useful for the next lemma. Define
\[J_x(s) := \Gamma_x'(s) = \frac{2|x|}{\sqrt{1 + 4s}} = \frac{2|x|^2}{\Gamma_x(s)}.\]
We also define for $\bm{x}\in \R^{kn}$, $\Gamma_{\bm{x}}: (-1/4,\infty)^k \to (0,\infty)^k$ and $J_{\bm{x}}: (-1/4,\infty)^k \to \R$ as 
\[\Gamma_{\bm{x}}(\bm{s}) :=\bigl(\Gamma_{x_1}(s_1),\dots,\Gamma_{x_k}(s_k)\bigr)
\]
and
\[J_{\bm{x}}(\bm{s}) := J_{x_1}(s_1) \cdots J_{x_k}(s_k).\]

Finally, for \(\bm t\in(0,\infty)^k\), set

\begin{equation}\label{def: weightW}
W(\bm{x},\bm t)
:=
\prod_{j=1}^k J_{x_j}(\Gamma_{x_j}^{-1}(t_j))
=
\frac{2^k|x_1|^2\cdots |x_k|^2}{t_1\cdots t_k}.
\end{equation}


The following theorem relates the $L^2$ integral of $\mc{A}_{\bm{t}}f(\bm{x})$ with respect to the weight $W$ to a weighted $L^2$ integral of the Fourier extension operator $\mc{E}f$.

\begin{lemma}\label{lemma: equalityreplacingliuidentity}
      Let $n\geq 2$ and $1 \leq k < n$. For any $\bm{x}=(x_1,\dots ,x_k)\in \textup{Aff}$, with $x_i\neq 0$ for all $1\leq i\leq k$, define $y_j:=\pi^{-1}(x_j)$ for all $1\leq j\leq k$. Then, 
     \[\int_{(0,\infty)^{k}}|\mc{A}_{\bm{t}}f(\bm{x})|^2W(\bm{x}, \bm{t}) \, d\bm{t} = \int_{\R^k} |\mc{E}f(s_1y_1 + \cdots + s_ky_k)|^2 \, d\bm{s},\]
     for any $f\in C^{\infty}_0(B^n(0,1))$, where $W(\bm{x},\bm{t})$ is defined as in (\ref{def: weightW}). 
\end{lemma}

\begin{remark}
Before proving this lemma, we note that for $k=1$, it follows that
 \[\int_{0}^{\infty}|\mc{A}_{t}f(x_1)|^2W(x_1,t) \, dt = \int_{\R} |\mc{E}f(s_1y_1)|^2 \, d s_1.\]
We can compare this with Liu's identity, which gives 
$$\int_{0}^{\infty} |\sigma_t*f(x_1)|^2 t^{n-1}dt=\int_{0}^{\infty} |\widehat{\sigma}_{s_1}*f(x_1)|^2{s_1}^{n-1}ds_1=\int_{0}^{\infty} |\mathcal{E}_{s_1\mathbb{S}^{n-1}}(\widecheck{f})(x_1)|^2s_1^{n-1}ds_1,$$
where $\mathcal{E}_{s_1\mathbb{S}^{n-1}}(g)(x)=\int g(w)e^{-2\pi i x\cdot w}d\tilde{\sigma}_{s_1}(w),$ for $\tilde{\sigma}_{s_1}$ the normalized surface measure in the sphere of radius $s_1$ in $\R^n$. We observe two key differences. One is that the extension operator in our identity is the extension operator associated to the truncated paraboloid $\mathbb{P}^n\subset \R^{n+1}$, while Liu's identity leads to an extension estimate for an $(n-1)$-dimensional sphere in $\R^n$. Secondly, our extension operator is acting on $f$, while in Liu's identity the extension operator acts on $\widecheck{f}$. 
\end{remark}

\begin{proof}[Proof of Lemma \ref{lemma: equalityreplacingliuidentity}]
    The idea of the proof can essentially be summarized as: perform a change of variables, then apply Plancherel, then apply the reverse change of variables as before. With $\bm{x}=(x_1,x_2, \dots ,x_k)$ fixed, we compute
\begin{align*}
    \int_{(0,\infty)^k}|\mc{A}_{\bm{t}}f(\bm{x})|^2W(\bm{x}, \bm{t}) \, d\bm{t} & = \int_{(0,\infty)^k}|\mc{A}_{\bm{t}}f(\bm{x})|^2J_{x_1}(\Gamma_{x_1}^{-1}(t_1)) \cdots J_{x_k}(\Gamma_{x_k}^{-1}(t_k)) \, d\bm{t} \\ &= \int_{(-1/4,\infty)^k}|\mc{A}_{\Gamma_{\bm{x}}(\bm{s})}f(\bm{x}) J_{\bm{x}}(\bm{s})|^2\, d\bm{s}
\end{align*}
by making the change of variables $s_j = \Gamma_{x_j}^{-1}(t_j)$.

We now take the Fourier transform in $\bm{s}$ of the function 
\[\bm{s} \mapsto \mc{A}_{\Gamma_{\bm{x}}(\bm{s})}f(\bm{x})J_{\bm{x}}(\bm{s})1_{(-1/4,\infty)^k}(\bm{s})\]
and apply Plancherel. Indeed, for $\bm{\xi}\in \R^k$ we compute 
\begin{equation}\label{eq: sphericalint}
\begin{split}
& \int_{(-1/4,\infty)^k}\mc{A}_{\Gamma_{\bm{x}}(\bm{s})}f(\bm{x})J_{\bm{x}}(\bm{s})e^{-2\pi i \bm{s} \cdot \bm{\xi}} \, d\bm{s}  \\ & = \int_{(-1/4,\infty)^k}\int_{S^{n - k}(\bm{x}, \Gamma_{\bm{x}}(\bm{s}))}f(x_{k+1})J_{\bm{x}}(\bm{s})e^{-2\pi i \bm{\xi} \cdot \bm{s}} \, d\sigma_{\bm{x},\Gamma_{\bm{x}}(\bm{s})}(x_{k+1}) \, d\bm{s}.
\end{split}
\end{equation}

Note that in the inner integral, we are integrating over $x_{k+1}$ satisfying $|x_j-x_{k+1}|=\Gamma_{x_j}(s_j)$ for all $1\leq j\leq k$. Now we invoke equality (\ref{distancedotproductrelation}) from \cite{IPPS22} to write,
\[|x_j - x_{k+1}| = \Gamma_{x_j}(s_j) \quad \Leftrightarrow \quad y_j \cdot (x_{k+1}, |x_{k+1}|^2) = s_j\]
where $y_j = \pi^{-1}(x_j)$.

So by substituting this expression for $s_j$ in (\ref{eq: sphericalint}), 
\begin{align*}
    = \int_{(-1/4,\infty)^k}\int_{S^{n - k}(\bm{x}, \Gamma_{\bm{x}}(\bm{s}))}f(x_{k+1})J_{\bm{x}}(\bm{s})e^{-2\pi i (\xi_1y_1 + \cdots + \xi_ky_k) \cdot (x_{k+1}, |x_{k+1}|^2)} \, d\sigma_{\bm{x},\Gamma_{\bm{x}}(\bm{s})}(x_{k+1}) \, d\bm{s}.
\end{align*}
Making the change of variables $t_j = \Gamma_{x_j}(s_j)$, this is equal to 
\begin{align*}
    \int_{(0,\infty)^{k}}\int_{S^{n - k}(\bm{x}, \bm{t})}f(x_{k+1})&e^{-2\pi i (\xi_1y_1 + \cdots + \xi_ky_k) \cdot (x_{k+1}, |x_{k+1}|^2)}\, d\sigma_{\bm{x},\bm{t}}(x_{k+1}) \, d\bm{t} \\
    & = \int_{\R^n}f(x_{k+1})e^{-2\pi i (x_{k+1}, |x_{k+1}|^2) \cdot (\xi_1y_1 + \cdots + \xi_ky_k)} \, dx_{k+1} \\
    & = \mc{E}f(-(\xi_1y_1 + \cdots + \xi_ky_k)).
\end{align*}

Therefore, by Plancherel's theorem in the \(\bm{s}\)-variable,
\[
\int_{(0,\infty)^k}
|\mc{A}_{\bm{t}}f(\bm{x})|^2W(\bm{x},\bm{t})\,d\bm{t}
=
\int_{\R^k}
\left|
\mc{E}f\bigl(-(\xi_1y_1+\cdots+\xi_ky_k)\bigr)
\right|^2\,d\bm{\xi}.
\]
Changing variables \(\bm\xi\mapsto-\bm\xi\) gives the desired identity.
\end{proof}

\medskip

Averaging the identity in Lemma \ref{lemma: equalityreplacingliuidentity} over the pins leads to a global weighted extension estimate, where the weight is explicitly
described as a convolution of measures supported on families of lines determined by the
pinned sets. That is precisely described in the theorem below.

\begin{theorem}\label{thm:parab_extension}
    Assume $n\geq 2$ and $1 \leq k < n$. Let $\mu_1, ..., \mu_k$ be finite compactly supported Borel measures on $\R^n\setminus\{0\}$ and let 
    $\bm{\mu} := \mu_1 \times \cdots \times \mu_k$. Suppose that $\bm{\mu}((\textup{Aff})^c) = 0$. Then for any $f\in C^{\infty}_0(B^n(0,1))$,
    \[\int \int_{(0,\infty)^k}|\mc{A}_{\bm{t}}f(\bm{x})|^2W(\bm{x}, \bm{t}) \, d\bm{t} \, d\bm{\mu}(\bm{x}) = \int_{\R^{n+1}}|\mc{E}f(y)|^2 dw(y)\]
    where $w = w_{k,\mu_1, ..., \mu_k}$ is the measure defined by the $k$-fold convolution
    \[w = w_1 * \cdots * w_k\]
    and $w_j := \Phi_*(\mu_j \times dt)$, and the map $\Phi: (\R^n \setminus \{0\}) \times (\R \setminus \{0\}) \to (\R^n \setminus \{0\}) \times (\R \setminus \{0\})$ is defined by 
\[\Phi(x, t) := t\cdot (\pi^{-1})(x)=\frac{t}{4|x|^2}(-2x, 1).\]
\end{theorem}
In the applications of this theorem later, the measures \(\mu_j\) will be Frostman measures restricted to subsets of \(E\) from which the pins are chosen. After a suitable localization, the weight \(W(\bm{x},\bm{t})\) will be comparable to \(1\) on the relevant region. This description provides enough
Frostman structure in the weight to apply weighted restriction estimates, leading to
the improved threshold in Theorem~\ref{thm:main}.

\begin{remark}
    One can check that $\Phi$ is a diffeomorphism of $(\R^n \setminus \{0\}) \times (\R \setminus \{0\})$ with inverse given by
\[\Phi^{-1}(x, t) = \frac{1}{t}\left(-\frac{x}{2}, |x|^2\right).\]
See \cref{fig:phi_grid} for a visualization of the map $\Phi$. Note that $\Phi$ sends vertical lines to lines through the origin, and sends horizontal lines to parabolas. More precisely, $\Phi$ maps the vertical line $x = b$ to the line through the origin and $(-2b, 1)$, and maps the horizontal line $t = c$ to the paraboloid $t = |x|^2/c$.
\end{remark}

\begin{figure}[htbp]
\centering
\resizebox{\textwidth}{!}{%
\begin{tikzpicture}

\begin{groupplot}[
    group style={group size=2 by 1, horizontal sep=2.6cm},
    width=8.4cm,
    height=6.8cm,
    axis lines=middle,
    xlabel style={at={(ticklabel* cs:1)},anchor=west},
    ylabel style={at={(ticklabel* cs:1)},anchor=south},
    ticklabel style={font=\small},
    label style={font=\small},
    title style={font=\small},
    clip=true
]

\nextgroupplot[
    xlabel={$x$},
    ylabel={$t$},
    xmin=-2.2, xmax=2.2,
    ymin=0, ymax=4.2,
    xtick={-2,-1,0,1,2},
    ytick={0,1,2,3,4},
    minor tick num=1
]

\foreach \b in {-2,-1.5,-1,-0.5,0.5,1,1.5,2} {
    \addplot[gray!75, thin] coordinates {(\b,0) (\b,4.2)};
}

\foreach \c in {0.5,1,1.5,2,2.5,3,3.5,4} {
    \addplot[gray!75, thin, densely dashed] coordinates {(-2.2,\c) (2.2,\c)};
}

\addplot[gray!95, very thick] coordinates {(1,0) (1,4.2)};
\addplot[gray!65, very thick, densely dashed] coordinates {(-2.2,1) (2.2,1)};

\nextgroupplot[
    xlabel={$x$},
    ylabel={$t$},
    xmin=-2.4, xmax=2.4,
    ymin=0, ymax=4.2,
    xtick={-2,-1,0,1,2},
    ytick={0,1,2,3,4}
]

\foreach \b in {-2,-1.5,-1,-0.5,0.5,1,1.5,2} {
    \addplot[gray!75, thin, domain=-2.4:2.4, samples=2] {-x/(2*\b)};
}

\foreach \c in {0.5,1,1.5,2,2.5,3,3.5,4} {
    \addplot[gray!75, thin, densely dashed, domain=-2.4:2.4, samples=250] {x^2/(\c)};
}

\addplot[gray!95, very thick, domain=-2.4:2.4, samples=2] {-x/2};
\addplot[gray!65, very thick, densely dashed, domain=-2.4:2.4, samples=250] {x^2};

\end{groupplot}

\draw[-{Latex[length=3mm]}, very thick]
  ($(group c1r1.east)+(7mm,0)$) -- ($(group c2r1.west)+(-7mm,0)$)
  node[midway, above=5pt] {$\Phi$};
\end{tikzpicture}%
}
\caption{A two-dimensional schematic of the map \(\Phi\). Vertical lines are mapped to lines through the origin, while horizontal lines are mapped to parabolas.}\label{fig:phi_grid}
\end{figure}
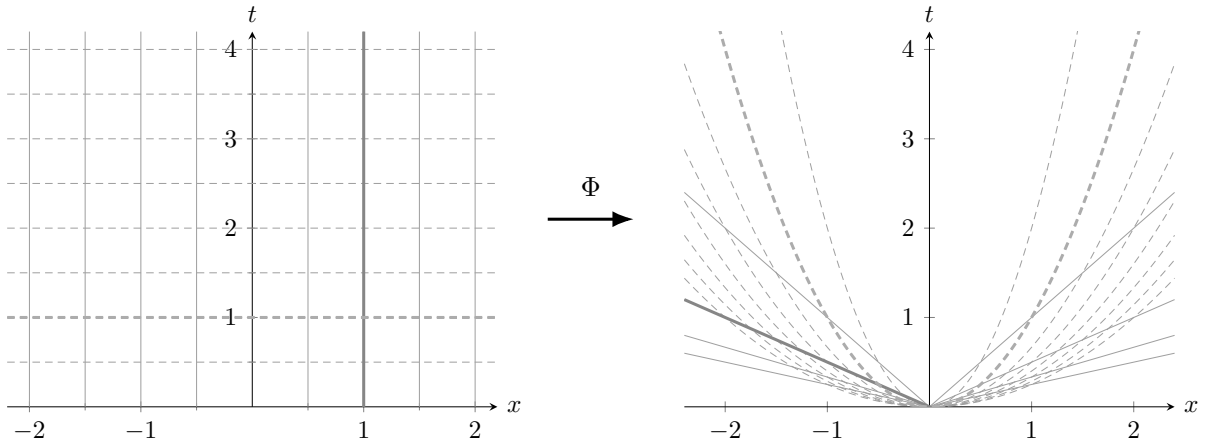

\begin{remark}
The measure \(w_j\) appearing in \cref{thm:parab_extension} can be viewed as an extension of \(\mu_j\) to \(\R^{n+1}\). For \(y\in\R^{n+1}\setminus\{0\}\), let
\[
L_y:=\{ry:r\in\R\}
\]
be the line through the origin in the direction of \(y\). If \(E_j=\operatorname{supp}\mu_j\), then
\[
\operatorname{supp} w_j
=
\Phi(E_j\times\R)
=
\bigcup_{x\in E_j}L_{(-2x,1)}.
\]
Thus \(w_j\) is supported on a union of lines through the origin determined by \(E_j\).
\end{remark}

\begin{proof}[Proof of \cref{thm:parab_extension}]

Let $\nu_j := (\pi^{-1})_* \mu_j$ denote the pushforward of $\mu_j$ under the map $\pi^{-1}$. Then by integrating the equality in Lemma \ref{lemma: equalityreplacingliuidentity} in $\bm{x}$,
\begin{align*}
    &\int \int_{(0,\infty)^k}|\mc{A}_{\bm{t}}f(\bm{x})|^2W(\bm{x}, \bm{t}) \, d\bm{t} \, d\bm{\mu}(\bm{x}) \\ &= \int\int_{\R^k} |\mc{E}f(s_1\pi^{-1}(x_1) + \cdots + s_k\pi^{-1}(x_k))|^2 \, d\bm{s} \, d(\mu_1 \times \cdots \times \mu_k)(\bm{x}) \\
    & = \int\int_{\R^k} |\mc{E}f(s_1y_1 + \cdots + s_ky_k)|^2 \, d\bm{s} \, d(\nu_1 \times \cdots \times \nu_k)(\bm{y}) \\
    & = \int_{\R^{(n+1)k}} |\mc{E}f(z_1 + \cdots + z_k)|^2 \, d(\Psi_*(\nu_1 \times ds_1) \times \cdots \times \Psi_*(\nu_k \times ds_k))(\bm{z}) \\ & = \int_{\R^{n+1}} |\mc{E}f(y)|^2 \, d(\Psi_*(\nu_1 \times ds_1) * \cdots * \Psi_*(\nu_k \times ds_k))(y) 
\end{align*}
where $\Psi: \R^{n + 1}\times \R \to \R^{n + 1}$ is defined by 
\[\Psi(y, s) = sy.\]
Using that $\nu_j \times ds_j = (\pi^{-1} \times \text{id})_{*}(\mu_j \times ds_j)$, we can write
\[\Psi_*(\nu_j\times ds_j)=(\Psi\circ (\pi^{-1}\times \text{id}))_{*}(\mu_j\times d s_j).\]
Finally we can compute
\[\Psi \circ (\pi^{-1} \times \text{id})(x, t) = \frac{t}{4|x|^2}(-2x, 1) = \Phi(x, t)\]
and we arrive at $\Psi_*(\nu_j\times ds_j)= \Phi_*(\mu_j \times dt)=w_j$ as wanted.
\end{proof}

\section{Facts About the Weights}\label{sec:facts_weight}
In this section, we will study the weight $w$ defined in \cref{thm:parab_extension}. We will now specialize to the case where $\mu_1, ..., \mu_k$ are $\alpha$-Frostman measures satisfying the support condition 
\[\text{supp} \, \mu_j \subset \{x \in \R^n : 1/2 \leq |x| \leq 1\}\]
as well as the transversality condition: 
\begin{equation}\label{eq:transversalityassumption}
   |x_1 \wedge \cdots \wedge x_k| \gtrsim 1, \quad \text{for all } x_j \in E_j := \text{supp} \, \mu_j. 
\end{equation}

 Thus \eqref{eq:transversalityassumption} is a quantitative linear independence condition on the sets \(E_1,\dots,E_k\). When dealing with distance graphs in $E\subset \R^n$, we can usually reduce to this scenario due to Lemma \ref{lemma: transversalityreduction}. Note that the condition above implies in particular $\bm{\mu}((\text{Aff})^c) = 0$.

The measures $w_j$ obtained in this case can be seen as a $(n+1)$-dimensional extension of the $\alpha$-Frostman measures $\mu_j$. We will indeed show in particular that the measures $w_j$ and $w$ can be decomposed into a sum of $(\alpha + 1)$-Frostman measures. \par
We fix dyadic scales $R_0, R_1, R_2, \cdots$ where $R_i = 2^i$. Let $A_R$ denote the annulus 
\[A_R := \{y \in \R^{n + 1} : R / 2 \leq |y| \leq R\}\]
and let $B_0$ denote the unit ball
\[B_0 := \{y \in \R^{n + 1} : |y| \leq 1\}.\]

\begin{theorem}\label{thm:k_weight_frostman}
    Let $\mu_1, ..., \mu_k$ be $\alpha$-Frostman measures satisfying the support condition 
    \[\text{supp} \, \mu_j \subset \{x \in \R^n : 1/2 \leq |x| \leq 1\}\]
    and \cref{eq:transversalityassumption} and let $w_j$, $w$ be the measures defined in \cref{thm:parab_extension}. I.e., 
    \[w_j = \Phi_*(\mu_j \times dt), \quad w = w_1 * \cdots * w_k, \quad \Phi(x, t) = \frac{t}{4|x|^2}(-2x, 1).\]
    Let $w^{0}$ be the restriction of $w$ to $B_0$ and for each $i \geq 1$, let $w^{i}$ denote the restriction of $w$ to $A_{R_i}$ rescaled by $R_i^{\alpha - k + 1}$. I.e.
    \[w^0 := w|_{B_0}, \quad w^i := R_i^{\alpha - k + 1} w|_{A_{R_i}}.\]
    Then $w^0(B_0) \lesssim 1$ and $w^i$ is an $(\alpha + 1)$-Frostman measure on $B^{n+1}(0, R_i)$. In other words, we have a decomposition of $w$ into a sum of $(\alpha+1)$-Frostman measures:
    \[w = w^0 + \sum_{i = 1}^{\infty}R_i^{-\alpha + k - 1}w^i.\]
\end{theorem}

Before proving the theorem above, we will need a few lemmas. For the remainder of this section, we will fix $\mu_1, \cdots, \mu_k$ and $w, w^i$ as in the statement of the theorem. We also define for each $1\leq j\leq k$ the measures $w_{j}^i$ from $w_j$ analogously to $w^i$:
\begin{equation}\label{eq:piecesofwj}
  w_{j}^0 := w_j|_{B_0}, \quad w_{j}^i := R_i^{\alpha} w_j|_{A_{R_i}}.  
\end{equation}

Therefore for each $j$, we have a decomposition
\[w_j = w_j^0 + \sum_{i = 1}^{\infty}R_i^{-\alpha}w_{j}^i.\]
We start by stating a simple lemma about the support of the measures $w_j$. Observe that
\[\Phi^{-1}(x, t) = \frac{1}{t}\left(-\frac{x}{2}, |x|^2\right).\]
This transformation is singular when $t \gg |x|$, but the next lemma says we don't need to worry about this region.

\begin{lemma}
   Let $S_j := \text{supp} \, w_j$. Then $S_j \subset \{y = (x, t) \in \R^{n + 1} : |x| \approx |t|\}$.
\end{lemma}

\begin{proof}
    The support of $\mu_j \times dt$ is contained in 
    \[\{x \in \R^n : 1/2 \leq |x| \leq 1\} \times \R.\]
    Therefore, the support of $w_j$ is contained in 
    \[\Phi\left(\{x \in \R^n : 1/2 \leq |x| \leq 1\} \times \R\right).\]
    Since $(x', t') := \Phi(x, t) = \frac{t}{4|x|^2}(-2x, 1)$ and 
    \[|x'| = \left|\frac{t}{4|x|^2}(-2x)\right| = \frac{|t|}{2|x|} \approx \frac{|t|}{4|x|^2} = |t'|,\]
    we have that $S_j \subset \{y = (x, t) \in \R^{n + 1} : |x| \approx |t|\}$ as wanted.
\end{proof}
The next two lemmas show that each $w_j^i$ is an $(\alpha + 1)$-Frostman measure on $A_{R_i}$. We define
\[\Phi^{-1}(x, t) = \frac{1}{t}\left(-\frac{x}{2}, |x|^2\right) =: (\varphi_1(x, t), \varphi_2(x, t)).\]
\begin{lemma}\label{lem:map_lipschitz}
    For $R \geq 1$, $\varphi_1|_{S_j \cap A_R}$ is Lipschitz with constant approximately $R^{-1}$ and $\varphi_2|_{S_j \cap A_R}$ is Lipschitz with constant approximately $1$. I.e. for $y,y' \in S_j \cap A_R$, we have
    \[|\varphi_1(y) - \varphi_1(y')| \lesssim R^{-1}|y - y'|, \quad |\varphi_2(y) - \varphi_2(y')| \lesssim |y - y'|.\]
\end{lemma}

We omit the proof of this lemma since it is a straightforward computation. 

\begin{lemma}\label{lem:1_weight_frostman}
    For all $R \geq 1$, we have that $w_j(B^{n+1}(0, R)) \lesssim R$ and for each $i \geq 1$, the weight $w_j^i$ defined in equation (\ref{eq:piecesofwj}) is an $(\alpha + 1)$-Frostman measure on $B^{n+1}(0, R_i)$. In other words, we have a decomposition of $w_j$ into a sum of Frostman measures:
    \[w_j = w_j^0 + \sum_{i = 1}^{\infty}R_i^{-\alpha }w_j^i.\]
\end{lemma}

\begin{proof}
    For the first assertion, note that for $y = (x,t) \in S_j$, we have that 
    \[|\Phi^{-1}(x, t)| \lesssim 1 + |x|.\]
    Therefore 
    \[w_j(B^{n+1}(0, R)) \leq (\mu_j \times dt)(B^{n+1}(0, CR)) \lesssim R\]
    which is the first assertion. For the second assertion, since \(w_j^i\) is supported on \(A_{R_i}\), it suffices to estimate
\(w_j^i(B(y,r))\) for balls \(B(y,r)\) with \(y\in A_{R_i}\). Moreover, if \(r\gtrsim R_i\), then
\[
w_j^i(B(y,r))
\leq R_i^\alpha w_j(B(0,CR_i))
\lesssim R_i^{\alpha+1}
\lesssim r^{\alpha+1}.
\]
Thus we may assume \(r\leq cR_i\).

    By Lemma \ref{lem:map_lipschitz} $\Phi^{-1}(B^{n+1}(y, r)\cap A_{R_i})$ is a tube  $T$ of dimensions approximately 
    \[(r / R_i) \times \cdots \times (r / R_i) \times r\] 
    pointing in the direction of the $t$-axis. Therefore
    \[w_j(B^{n+1}(y, r)\cap A_{R_i}) \leq (\mu_j \times dt)(T) \lesssim (r/R_i)^{\alpha}r = R_i^{-\alpha}r^{\alpha + 1}.\]
    Therefore
    \[
w_j^i(B^{n+1}(y,r))
=
R_i^{\alpha}w_j(B^{n+1}(y,r)\cap A_{R_i})
\lesssim r^{\alpha+1},
\]
    as wanted.
\end{proof}

It follows from the hypothesis on $E_j$ that they are separated and linearly independent in the sense that if $x_j \in E_j$ for $1 \leq j \leq k$, then $\{x_j\}_{j=1}^{k}$ is linearly independent. Recall that $S_j=\text{supp} (w_j)\subset \Phi(E_j\times \R)$ for 
\[\Phi(x, t) = \frac{t}{4|x|^2}(-2x, 1).\]
Define $C_j=\rho(S_j\setminus\{0\})$, where $\rho$ denotes the radial projection map defined by 
\[\rho: \R^{n+1} \setminus \{0\} \to \mathbb{S}^n, \quad \rho(y) = \frac{y}{|y|}.\] 
We claim that the sets $\{C_j\}_{j=1}^{k}$ are linearly independent. Indeed, every \(z_j\in C_j\) is of the form
\[
z_j=\pm \frac{(-2x_j,1)}{\sqrt{4|x_j|^2+1}},
\qquad x_j\in E_j.
\]
Thus any nontrivial linear relation among \(z_1,\dots,z_k\) would, by looking at the first \(n\) coordinates, give a nontrivial linear relation among \(x_1,\dots,x_k\), contradicting the wedge transversality assumption.

\par
Since $w$ is a convolution of the $w_j$, it is important to establish some transversality properties of the supports of the $w_j$.

\begin{lemma}\label{lem:weight_support}
    For $y_j \in S_j=\textup{supp}\, w_j$, we have \[|y_1| + \cdots + |y_k| \lesssim |y_1 + \cdots + y_k|.\]
\end{lemma} 

\begin{proof}
    Let 
    \[A := \left\{a \in \R^{k} : \sum_{j = 1}^{k}|a_j| = 1\right\}.\]
    Since for any $z_j \in C_j$, the $\{z_j\}$ are linearly independent, the map 
    \[C_1 \times \cdots \times C_k \times A \to \R^{n+1}, \qquad (z_1, \ldots, z_k, a) \mapsto a_1z_1 + \cdots + a_kz_k\]
    is never zero. Thus by compactness, we have that 
    \[|a_1z_1 + \cdots + a_kz_k| \gtrsim 1, \quad \text{ for all } (z_1, \ldots, z_k, a) \in C_1 \times \cdots \times C_k \times A.\] 

    If \(y_1=\cdots=y_k=0\), there is nothing to prove. Otherwise, set
    \[
    a_j := \frac{|y_j|}{|y_1|+\cdots+|y_k|}.
    \]
    For each \(j\) with \(y_j\neq0\), set \(z_j:=y_j/|y_j|\in C_j\). For indices with \(y_j=0\), choose an arbitrary \(z_j\in C_j\). Then \(a\in A\), and
    \[
    y_1+\cdots+y_k
    =
    (|y_1|+\cdots+|y_k|)
    (a_1z_1+\cdots+a_kz_k).
    \]
    Therefore
    \[
    |y_1+\cdots+y_k|
    \gtrsim
    |y_1|+\cdots+|y_k|,
    \]
    as desired.
\end{proof}
We are now ready to prove \cref{thm:k_weight_frostman}.

  \begin{proof}[Proof of Theorem \ref{thm:k_weight_frostman}]
The case \(k=1\) follows directly from Lemma~\ref{lem:1_weight_frostman}, so assume \(k\geq 2\).

We first prove the estimate on \(B_0\). If \(y_j\in S_j\) and
\(y_1+\cdots+y_k\in B_0\), then Lemma~\ref{lem:weight_support} gives
\[
|y_1|+\cdots+|y_k|\lesssim |y_1+\cdots+y_k|\lesssim 1.
\]
Hence \(|y_j|\lesssim 1\) for every \(j\), and Lemma~\ref{lem:1_weight_frostman} gives
\[
w(B_0)\leq \prod_{j=1}^k w_j(B(0,C))\lesssim 1.
\]
Thus \(w^0(B_0)\lesssim 1\).

Now fix \(i\geq 1\), write \(R=R_i\), and let \(B=B^{n+1}(y_0,r)\) be an arbitrary ball. Since \(w^i\) is supported on \(A_R\), it suffices to estimate \(w(B\cap A_R)\). If \(r\gtrsim R\), then
\[
w^i(B)
\leq R^{\alpha-k+1}w(B(0,CR))
\leq R^{\alpha-k+1}\prod_{j=1}^k w_j(B(0,CR))
\lesssim R^{\alpha+1}\lesssim r^{\alpha+1}.
\]
Thus we may assume \(r\leq cR\).

Suppose \(y_j\in S_j\) and \(y_1+\cdots+y_k\in B\cap A_R\). By Lemma~\ref{lem:weight_support},
\[
|y_1|+\cdots+|y_k|
\lesssim |y_1+\cdots+y_k|
\lesssim R,
\]
while the reverse inequality
\[
R\lesssim |y_1+\cdots+y_k|\leq |y_1|+\cdots+|y_k|
\]
is immediate from \(y_1+\cdots+y_k\in A_R\). Hence
\[
|y_1|+\cdots+|y_k|\sim R.
\]
In particular, at least one \(y_\ell\) satisfies \(|y_\ell|\gtrsim R\). Therefore, we can decompose
\begin{align*}
    w(B\cap A_R) & = \int 1_{\{y_1 + \cdots + y_k \in B \cap A_R\}}(y_1, ..., y_k) \, dw_1(y_1) \cdots dw_k(y_k) \\
    & \leq \sum_{l = 1}^{k}\int 1_{\{y_1 + \cdots + y_k \in B \cap A_R, |y_l| \gtrsim R\}}(y_1, ..., y_k) \, dw_1(y_1) \cdots dw_k(y_k)
\end{align*}
By symmetry, it suffices to estimate the contribution of the piece where \(|y_1|\gtrsim R\), up to a harmless factor of \(k\). On this piece, also \(|y_j|\lesssim R\) for every \(j\). Therefore
\[
w(B\cap A_R)
\lesssim
\int_{\{|y_2|,\dots,|y_k|\lesssim R\}}
w_1\bigl((B-(y_2+\cdots+y_k))\cap\{z:|z|\sim R\}\bigr)
\,dw_2(y_2)\cdots dw_k(y_k).
\]
For fixed \(y_2,\dots,y_k\), the set inside \(w_1\) is contained in a ball of radius \(r\), intersected with \(O(1)\) dyadic annuli of radii comparable to \(R\). By Lemma~\ref{lem:1_weight_frostman},
\[
w_1\bigl((B-(y_2+\cdots+y_k))\cap\{z:|z|\sim R\}\bigr)
\lesssim R^{-\alpha}r^{\alpha+1}.
\]
Using \(w_j(B(0,CR))\lesssim R\) for \(2\leq j\leq k\), we obtain
\[
w(B\cap A_R)
\lesssim
R^{-\alpha}r^{\alpha+1}R^{k-1}
=
R^{-\alpha+k-1}r^{\alpha+1}.
\]
Multiplying by \(R^{\alpha-k+1}\), we get
\[
w^i(B)\lesssim r^{\alpha+1}.
\]
Thus \(w^i\) is an \((\alpha+1)\)-Frostman measure.
\end{proof}

\section{Weighted Fourier Extension Estimates}\label{ref:ProofofThmA}

In this section, we prove \cref{thm:main}. To summarize, we use the results of the previous section to reduce the proof of \cref{thm:main} to a local weighted Fourier extension estimate involving the weight $w$. We can then apply known weighted restriction estimates to prove \cref{thm:main}. \par
For any $R > 0$, let $\mc{P}_{\leq R}$ and $\mc{P}_{\geq R}$ denote the Littlewood-Paley projections onto frequencies $\{|\xi| \leq R\}$ and $\{|\xi| \geq R\}$ respectively. We start by stating an elementary lemma.
\begin{lemma}\label{lem:stationary_phase}
    Let $R \geq 1$ and $f$ be a function with support in the unit ball of $\R^n$. Then for any $s, N > 0$
    \[\sup_{|y| \leq R} |\mc{E}(\mc{P}_{\geq 4R}f)(y)| \lesssim_{s, N} R^{-N}\|f\|_{H^{-s}}.\]
\end{lemma}

\begin{proof}
    The proof is fairly standard but we include a sketch for completeness. Let $y=(x,t)$ with $|y| \leq R$ and consider the function
    \[a_y(z) = \chi(z)e^{2\pi i (x \cdot z + t|z|^2)}\]
    where $\chi$ is a smooth bump function equal to $1$ on $B^n(0, 1)$ and supported on $B^n(0, 5/4)$. Then 
    \begin{align*}
        |\mc{E}(\mc{P}_{\geq 4R}f)(y)| & = \left|\int a_y(z)\mc{P}_{\geq 4R}f(z) \, dz\right| \\
        & = \left|\int_{|\xi| \geq 4R} \widehat{a_y}(\xi)\widehat{f}(\xi) \, d\xi\right| \\
        & \leq \|f\|_{H^{-s}}\left(\int_{|\xi| \geq 4R} \ang{\xi}^{2s}|\widehat{a_y}(\xi)|^2 \, d\xi\right)^{1/2}.
    \end{align*}
    
    The claim will follow once we show that 
    \[\int_{|\xi| \geq 4R} \ang{\xi}^{2s}|\widehat{a_y}(\xi)|^2 \, d\xi \lesssim_{s, N} R^{-2N}.\]
    We can assume $R \geq 1$ (in the case $R<1$ we use $\|a_y\|_{H^s}\lesssim_s 1$). We write 
    \[\widehat{a_y}(\xi) = \int \chi(z)e^{i\Phi_{y,\xi}(z)} \, dz\]
    where $\Phi_{y, \xi}(z) = 2\pi(x \cdot z + t|z|^2 -z \cdot \xi)$. Since $|z| \leq 5/4$, $|y| \leq R$ and $|\xi| \geq 4R$, we have that
    \[|\nabla_z \Phi_{y, \xi}(z)| \gtrsim |\xi|.\]
    Applying integration by parts $M$ times, we have that
    \[|\widehat{a_y}(\xi)| \lesssim_{M} |\xi|^{-M}.\]
    Therefore 
    \[\int_{|\xi| \geq 4R} \ang{\xi}^{2s}|\widehat{a_y}(\xi)|^2 \, d\xi \lesssim_M \int_{|\xi| \geq 4R} |\xi|^{2s - 2M} \, d\xi \lesssim_{s, N} R^{-2N}\]
    by choosing $M$ sufficiently large depending on $s$ and $N$.
\end{proof}

\begin{lemma}[Transversality Lemma]
\label{lemma: transversalityreduction} 
Let $1\leq k\leq n$ and let \(E\subset \R^n\) be a compact set. Let \(\mu\) be a nonzero
\(s\)-Frostman measure supported on \(E\), and assume \(s\in (k-1,n]\).
Let \(F_1,\dots,F_k\subset \text{supp}(\mu)\) be \(\mu\)-measurable pairwise disjoint sets with
\(\mu(F_i)>0\) for each \(i\).

Then there exist compact sets \(E_i\subset F_i\), \(i=1,\dots,k\),
with \(\mu(E_i)>0\), and a constant \(c>0\), such that
\[
 |x_1\wedge\cdots\wedge x_k|\geq c
 \qquad\text{for all }x_i\in E_i.
\]
\end{lemma}

\begin{remark}
    In the \cite{BFOPR2026} machinery for $k$ admissible graphs, it suffices to have the $L^2$ $k$-star building blocks for pins coming from transverse pieces. The reason is that any sufficient threshold $\alpha_{+}^{*}(n,k)\in (0,n)$ valid for positive measure of $k$ stars is at least $k$ by \cite[Remark 7]{BFOPR2026} (so larger than $k-1$). Given a compact set $E$ of dimension larger than $\alpha_{+}^*(n,k)$, one takes a Frostman measure $\mu$ in $E$ with parameter $s>\alpha_{+}^{*}(n,k)$ and find separated $\mu$ positive measure compact pieces $F_1, F_2, \dots ,F_{|\mathcal{V}|}\subset E$ where $\mathcal{V}$ is the vertex set of the $k$ admissible graph we care for. Then one can successfully apply the lemma above to every $k$ out of $|\mathcal{V}|$ pieces. This only requires finitely many refining steps ($\binom{|\mathcal{V}|}{k}$ many steps) until we are reduced to pairwise disjoint compact pieces $\{E_i\colon 1\leq i\leq |\mathcal{V}|\}$ such that $\mu(E_i)>0$ for all $i$, and such that any $k$ pieces $\{E_{i_j}\}_{j=1}^k$, $1\leq i_1<i_2<\dots<i_k\leq |\mathcal{V}|$ are transverse. 
\end{remark}

\begin{proof}[Proof of Lemma \ref{lemma: transversalityreduction}]
We first record a simple consequence of the Frostman condition. If \(V\)
is any linear subspace of \(\R^n\) with \(\dim V\leq k-1\), then
\[
\mu(V)=0.
\]
Indeed, since \(\text{supp}(\mu)\) is compact, \(V\cap \text{supp}(\mu)\) can be covered
by \(O(r^{-\dim V})\) balls of radius \(r\). Hence, by the
\(s\)-Frostman condition,
$\mu(V)\lesssim r^{s-\dim V}.$
Since \(s>k-1\geq \dim V\), letting \(r\to0\) gives \(\mu(V)=0\).

Now define \(\mu_i:=\mu_{F_i}\) (restriction of $\mu$ to $F_i$). We claim that
\[
(\mu_1\times\cdots\times\mu_k)
\Bigl(\{(x_1,\dots,x_k):x_1\wedge\cdots\wedge x_k=0\}\Bigr)=0.
\]

Equivalently, the set of linearly dependent \(k\)-tuples has zero
\(\mu_1\times\cdots\times\mu_k\)-measure.

We prove this by induction. For \(1\leq m\leq k\), let
\[
Z_m:=\{(x_1,\dots,x_m): x_1,\dots,x_m
\text{ are linearly dependent}\}.
\]

For \(m=1\), \(Z_1=\{0\}\), and \(\mu_1(\{0\})=0\) by the observation above.

Assume the claim has been proved for \(m-1\). Then
\[
Z_m\subset (Z_{m-1}\times \R^n)
\cup
\{(x_1,\dots,x_m): x_1,\dots,x_{m-1}
\text{ are independent and }
x_m\in \text{span}(x_1,\dots,x_{m-1})\}.
\]
The first set has zero product measure by the induction hypothesis. For
the second set, fix an independent \((m-1)\)-tuple
\((x_1,\dots,x_{m-1})\). Then
\[
\text{span}(x_1,\dots,x_{m-1})
\]
is an \((m-1)\)-dimensional linear subspace. Since \(m-1\leq k-1\),
the preliminary observation gives
\[
\mu_m\bigl(\text{span}(x_1,\dots,x_{m-1})\bigr)=0.
\]
Fubini therefore gives that the second set also has zero product measure.
This proves the claim for \(m\), and hence for \(m=k\).

Since each \(\mu_i\) is nonzero, the product measure
\(\mu_1\times\cdots\times\mu_k\) has positive total mass. Therefore we
may choose
\[
(a_1,\dots,a_k)\in F_1\times\cdots\times F_k
\]
outside the exceptional set, with each \(a_i\in \text{supp}(\mu_i)\). Thus
\[
|a_1\wedge\cdots\wedge a_k|>0.
\]
By continuity of the map
\[
(x_1,\dots,x_k)\mapsto |x_1\wedge\cdots\wedge x_k|,
\]
there exist radii \(\rho_i>0\) and a constant \(c>0\) such that
\[
|x_1\wedge\cdots\wedge x_k|\geq c
\]
whenever \(x_i\in B(a_i,\rho_i)\) for every \(i\).

Since \(a_i\in \text{supp}(\mu_i)\), we have
\[
\mu_i(B(a_i,\rho_i))>0.
\]
Equivalently,
\[
\mu(F_i\cap B(a_i,\rho_i))>0.
\]
By inner regularity, choose compact sets
\[
E_i\subset F_i\cap B(a_i,\rho_i)
\]
with \(\mu(E_i)>0\). Then, for all \(x_i\in E_i\),
\[
|x_1\wedge\cdots\wedge x_k|\geq c.
\]

\end{proof}

\begin{theorem}\label{thm:pinned_thm}
    Let $\alpha > 0$. Suppose that for some $\gamma = \gamma(\alpha) \geq 0$, we have the estimate
    \[\int_{B^{n+1}(0, R)}|\mc{E}f(y)|^2 \, d\nu(y) \lesssim R^{\gamma}\|f\|_{L^2}^2\]
    for all $(\alpha + 1)$-Frostman measures $\nu$ in $\R^{n+1}$ and $f$ supported on the unit ball in $\R^n$. Then if $E\subset \R^n$ is a compact set with Hausdorff dimension 
    \[\dim(E)>\alpha > \frac{n + k - 1 + \gamma}{2},\]
    and $\mu$ is an $\alpha$-Frostman measure on $E$, then $E$ has an abundance of $L^2$ $k$-pinned stars relative to $\mu$ and in particular, for an abundance of pins $(x_1,x_2,\dots ,x_k)\in E^k$, one has that 
\[  \mathcal{L}^k \bigl(\Delta_{x_1,\dots,x_k}^{k\text{-star}}(E)\bigr)>0.
\]
   
\end{theorem}

\begin{proof}
    Let $\mu$ be an $\alpha$-Frostman measure on $E$. Let
\(E_1,\dots,E_{k+1}\subset E\setminus\{0\}\) be a collection of $k+1$ separated compact subsets of $E\setminus \{0\}$ with \(\mu(E_i)>0\) for each \(1\leq i\leq k+1\) and $\{E_1,E_2,\dots ,E_k\}$ transverse. More precisely, we are assuming the following properties
    
    \begin{enumerate}
    \item Positive measure: $\mu(E_j)>0$ for each $1\leq j\leq k+1$.
        \item Transversality: $|x_1 \wedge \cdots \wedge x_k| \gtrsim 1$ for all $x_j \in E_j$, $1\leq j\leq k$.
        
         \item Distance from the origin: For all $1\leq j\leq k+1$ $E_j \subset \{x \in \R^n : \epsilon\leq |x| \leq C\}$ where $\epsilon,C>0$. Up to a rescaling argument we can assume that $\epsilon=1/2$ and $C=1$, so that $E_j \subset \{x \in \R^n : 1/2\leq |x| \leq 1\}$.
    \end{enumerate}

In practice, collections of $k+1$ pieces with these properties always exist whenever $\dim(E)>k-1$. One first applies the standard separation Lemma in \cite[Lemma 3]{BIO2023} successively to get to disjoint $k+1$ compact pieces of positive measure and that do not intersect the origin (properties (2) and (3)), and then for the transversality property, one applies the transversality Lemma to these pieces \ref{lemma: transversalityreduction} (observe that it is relevant that $\alpha>\frac{n+k-1}{2}>k-1$ so that Lemma applies).
    
     For each $1\leq j\leq k+1$, let $\mu_j$ be corresponding $\alpha$-Frostman measures on $E_j$, that is, $\mu_j=\mu_{E_j}$. As before denote $\bm{\mu}=\prod_{j=1}^{k} \mu_j$. It suffices to show that 
    \begin{equation}\label{eq:distance_measure}
        \int \int_{\R^k}|\mc{A}_{\bm{t}}f(\bm{x})|^2\, d\bm{t} \, d\bm{\mu}(\bm{x}) \lesssim \int_{\R^{n}} |\xi|^{-\beta}|\widehat{f}(\xi)|^2 \, d\xi
    \end{equation}
    for all smooth compactly supported functions $f$ and some $\beta > n - \alpha$. Indeed, if this is the case we set $(\mu_{k+1})_R = (\mu_{k+1}) * \eta_{R}$ where 
    \[\eta_R(x) = R^{n}\eta(Rx)\]
    for some smooth compactly supported function $\eta$ with $\int \eta = 1$. Then by \eqref{eq:distance_measure}, we have that 
    \[\int \int_{\R^k}|\mc{A}_{\bm{t}}(\mu_{k+1})_R(\bm{x})|^2\, d\bm{t} \, d\bm{\mu}(\bm{x}) \lesssim \int_{\R^{n}} |\xi|^{-\beta}|\widehat{\mu_{k+1}}(\xi)|^2 |\widehat{\eta}(R\xi)|^2\, d\xi \leq \int_{\R^{n}} |\xi|^{-\beta}|\widehat{\mu_{k+1}}(\xi)|^2 \, d\xi < \infty\]
    since $\beta > n - \alpha$ and $\mu_{k+1}$ is $\alpha$-Frostman.
    Passing to the limit as \(R\to \infty\), we obtain
\[
\int_{E_1\times\cdots\times E_k}
\|(d_{\bm x})_*\mu_{k+1}\|_{L^2(\R^k)}^2
\,d\bm\mu(\bm x)<\infty.
\]
Hence
\[
(d_{\bm x})_*\mu_{k+1}\in L^2(\R^k)
\]
for \(\bm\mu\)-almost every \(\bm x\in E_1\times\cdots\times E_k\), which is the conclusion of the theorem. \par
    It remains to establish \eqref{eq:distance_measure}. By \cref{thm:parab_extension}, for all smooth and compactly supported functions $f$, we have that
    \[\int \int_{\R^k}|\mc{A}_{\bm{t}}f(\bm{x})|^2W(\bm{x}, \bm{t}) \, d\bm{t} \, d\bm{\mu}(\bm{x}) = \int_{\R^{n+1}}|\mc{E}f(y)|^2 dw(y)\]
    where the weight $w$ is defined in \cref{thm:parab_extension}. Using the compact support of $f$ and the $\mu_j$,
    \[\mc{A}_{\bm{t}}f(\bm{x}) = 0\]
    if $t_j \geq C$ for some $t_j$. Therefore we have that 
    \[W(\bm{x}, \bm{t}) = \frac{2^k|x_1|^2 \cdots |x_k|^2}{t_1 \cdots t_k} \gtrsim 1\]
    in the integration region. In particular,
    \[\int \int_{\R^k}|\mc{A}_{\bm{t}}f(\bm{x})|^2 \, d\bm{t} \, d\bm{\mu}(\bm{x}) \lesssim \int_{\R^{n+1}}|\mc{E}f(y)|^2 dw(y).\]
    By \cref{thm:k_weight_frostman}, we can write 
    \[w = w^0 + \sum_{i = 1}^{\infty}R_i^{-\alpha + k - 1} w^{i}\]
    where $w^0(B_0) \lesssim 1$ and each $w^i$ is an $(\alpha + 1)$-Frostman measure on $B^{n+1}(0, R_i)$. Therefore we can write
    \[\int_{\R^{n+1}}|\mc{E}f(y)|^2 \, dw(y) \lesssim \int_{B_0}|\mc{E}f(y)|^2 \, dw^0(y) + \sum_{i = 1}^{\infty}R_i^{-\alpha + k - 1}\int_{A_{R_i}}|\mc{E}f(y)|^2 \, dw^i(y).\]
    Let $f_{\leq R} := \mc{P}_{\leq R}f$ and $f_{\geq R} := \mc{P}_{\geq R}f$. By Lemma \ref{lem:stationary_phase}, we have that
    \[|\mc{E}f(y) - \mc{E}f_{\leq 4R}(y)| \lesssim_{s, N} R^{-N}\|f\|_{H^{-s}}\]
    for all $y \in B^{n+1}(0, R)$ and any $s, N > 0$. Therefore if $N$ is large enough, the right hand side is bounded by a constant times
    \[\int_{B_0}|\mc{E}f_{\leq 4}(y)|^2 \, dw^0(y) + \sum_{i = 1}^{\infty}R_i^{-\alpha + k - 1}\int_{A_{R_i}}|\mc{E}f_{\leq 4R_i}(y)|^2 \, dw^i(y) + \|f\|_{H^{-s}}^2.\]
    To bound the first term, we can simply use 
    \begin{align*}
        \int_{B_0}|\mc{E}f_{\leq 4}(y)|^2 \, dw^0(y) & \leq \|\mc{E}f_{\leq 4}\|_{L^\infty(B_0)}^2w^0(B_0) \\
        & \lesssim \|\mc{E}f_{\leq 4}\|_{L^\infty(B_0)}^2 \\
        & \lesssim \left(\int |f_{\leq 4}(x)| \, dx\right)^2 \\
        & \lesssim \int_{|\xi| \leq 4}|\widehat{f}(\xi)|^2 \, d\xi \\
        & \leq \|f\|_{H^{-s}}^2
    \end{align*}
    To bound the second term, we can apply the assumption of the theorem for the $(\alpha + 1)$-Frostman measure $w_i$ to get that
    \[\int_{A_{R_i}}|\mc{E}f_{\leq 4R_i}(y)|^2 \, dw^i(y) \lesssim R_i^{\gamma}\|f_{\leq 4R_i}\|_{L^2}^2.\]
    By the assumption on $\alpha$, there is some small $\epsilon > 0$ so that
    \[\beta :=  \alpha - k + 1 - \gamma - \epsilon > n-\alpha .\]
    Therefore 
    \begin{align*}
        R_i^{-\alpha + k - 1}\int_{A_{R_i}}|\mc{E}f_{\leq 4R_i}(y)|^2 \, dw^i(y) & \lesssim R_i^{-\epsilon} R_i^{-\alpha + k - 1 + \gamma + \epsilon}\|f_{\leq 4R_i}\|_{L^2}^2 \\
        & \leq R_i^{-\epsilon}R_i^{-\beta} \int_{|\xi| \leq 4R_i}|\widehat{f}(\xi)|^2 \, d\xi \\
        & \lesssim R_i^{-\epsilon} \int_{|\xi| \leq 4R_i}|\widehat{f}(\xi)|^2 |\xi|^{-\beta} \, d\xi.
    \end{align*}
    By choosing $s$ large enough so that $2s > \beta$, we conclude \eqref{eq:distance_measure} as wanted.
\end{proof}

\cref{thm:main} is now a simple consequence of the theorem above and the local weighted Fourier extension estimate established in \cite{DZ2019}.

\begin{theorem}[Du--Zhang]\label{thm:du_zhang}
    Let $\nu$ be an $\eta$-Frostman measure on $B^{n+1}(0, R)$. For all $\epsilon > 0$, we have the estimate
    \[\int_{B^{n+1}(0, R)}|\mc{E}f(y)|^2 \, d\nu(y) \lesssim_\epsilon R^{\frac{\eta}{n + 1} + \epsilon}\|f\|_{L^2}^2.\]
\end{theorem}

\begin{proof}[Proof of \cref{thm:main}]
By \cref{thm:du_zhang}, applied with \(\eta=\alpha+1\), the hypothesis of
\cref{thm:pinned_thm} is satisfied with
\[
\gamma=\frac{\alpha+1}{n+1}+\epsilon
\]
for every \(\epsilon>0\). Since
\[
\alpha>\alpha_+(n,k)
=\frac{n^2+kn+k}{2n+1},
\]
we have
\[
2\alpha>n+k-1+\frac{\alpha+1}{n+1}.
\]
Choosing \(\epsilon>0\) sufficiently small, we obtain
\[
2\alpha>n+k-1+\frac{\alpha+1}{n+1}+\epsilon
=
n+k-1+\gamma.
\]
Equivalently,
\[
\alpha>\frac{n+k-1+\gamma}{2}.
\]
Therefore the assumptions of \cref{thm:pinned_thm} are satisfied, and hence \(E\) has an abundance of \(L^2\) \(k\)-pinned stars relative to \(\mu\).
\end{proof}

As a consequence of \cref{thm:parab_extension} as well as Liu's identity, we can derive a relationship between the Fourier extension operator for the paraboloid and the Fourier extension operator for the sphere. Indeed, let $\mu$ be a measure supported on the annulus $\{x \in \R^n : 1/2 \leq |x| \leq 1\}$ and let $f$ be a function supported on the unit ball with support positive distance from $\text{supp} \, \mu$. Then by \cref{thm:parab_extension}, for $k = 1$, we have 
\[\int_{\R^n} \int_{\R}|\mc{A}_{t}f(x)|^2\frac{2|x|^2}{t} \, dt \, d\mu(x) = \int_{\R^{n+1}}|\mc{E}_{\mathbb{P}^n}f(y)|^2 dw(y)\]
where 
\[w = \Phi_*(\mu \times dt), \quad \Phi(x, t) = \frac{t}{4|x|^2}(-2x, 1).\]
Since the supports of $\mu$ and $f$ are at positive distance from one another, we have that $t \approx 1$ for all $t$ in the integration region. Letting 
\[\sq{\mc{A}}_tf(x) := \int_{S^{n-1}(x, t)}f(y) \, d\sigma_{x,t}(y)\]
denote the \emph{normalized} spherical averaging operator, we have that
\[\int_{\R^{n+1}}|\mc{E}_{\mathbb{P}^n}f(y)|^2 dw(y) \approx \int_{\R^n} \int_{\R}|\sq{\mc{A}}_{t}f(x)|^2 t^{n-1} \, dt \, d\mu(x).\]
Now by applying Liu's identity, we have that
\begin{align*}
    \int_{\R^n} \int_{\R}|\sq{\mc{A}}_{t}f(x)|^2 t^{n-1} \, dt \, d\mu(x) & = \int_{\R^n} \int_{0}^{\infty}|\widehat{\sigma}_{r} * f(x)|^2 r^{n-1} \, dr \, d\mu(x) \\
    & = \int_{\R^n} \int_{0}^{\infty}|\mc{E}_{rS^{n-1}}\widehat{f}(x)|^2 r^{n-1} \, dr \, d\mu(x).
\end{align*}
Therefore we have the equivalence 
\[\int_{\R^{n+1}}|\mc{E}_{\mathbb{P}^n}f(y)|^2 dw(y) \approx  \int_{\R^n} \int_{0}^{\infty}|\mc{E}_{rS^{n-1}}\widehat{f}(x)|^2 r^{n-1} \, dr \, d\mu(x)\]
relating the Fourier extension operator for the paraboloid to the Fourier extension operator for the sphere. Now assume that $\mu$ is $\alpha$-Frostman. In the proof of \cref{thm:pinned_thm} for $k=1$, we showed that 
\[\int_{\R^{n+1}}|\mc{E}_{\mathbb{P}^n}f(y)|^2 dw(y) \lesssim_\epsilon \int_{\R^n}|\xi|^{-\alpha + \frac{\alpha + 1}{n + 1} + \epsilon}|\widehat{f}(\xi)|^2 \, d\xi.\]
This argument gives the threshold 
\[\alpha > \frac{n}{2} + \frac{1}{4} + \frac{3}{4(2n + 1)}\]
for the existence of pinned distance sets with positive measure. However, if one uses Du-Zhang's weighted restriction estimates for the $(n-1)$-dimensional sphere instead of the $n$-dimensional paraboloid (which is usually used to obtain results on Falconer's problem), one can show that
\[\int_{\R^{n+1}}|\mc{E}_{\mathbb{P}^n}f(y)|^2 dw(y) \lesssim_\epsilon \int_{\R^n}|\xi|^{-\alpha + \frac{\alpha}{n} + \epsilon}|\widehat{f}(\xi)|^2 \, d\xi\]
which is a bit better and gives
\[\alpha > \frac{n}{2} + \frac{1}{4} + \frac{1}{4(2n - 1)},\]
which was exactly Du and Zhang's threshold for the Falconer distance problem in $\R^n$ (see \cite[Theorem 2.6]{DZ2019}). An interesting question is to obtain better bounds for 
\[\int_{B^{n+1}(0, R)}|\mc{E}_{\mathbb{P}^n}\mu(y)|^2 d\nu(y)\]
when $\mu$ is an $\alpha_1$-Frostman measure on $B^n(0, 1)$ and $\nu$ is an $\alpha_2$-Frostman measure on $B^{n+1}(0, R)$. Following the above arguments in the case $k=1$, one can get slightly better bounds when the measure $\nu$ is the appropriate rescaled version of $w$ (in other words, $\nu$ is constant on lines through the origin). Obtaining improved bounds for general Frostman measures $\nu$ would allow us to improve results for $k$-stars, leading to improved results for graphs in general. 

\section{Nonempty Interior Results and Dimension Estimates}\label{sec:nonemptyanddimresults}

In this section, we adapt the methods used in the proof of the improved thresholds for positive measure of $k$-stars in Theorem \ref{thm:main} to get nonempty interior and dimension estimates for $k$-stars. We start by establishing Theorem \ref{thm: nonemptyinteriorforkstars}.

\begin{proof}[Proof of Theorem \ref{thm: nonemptyinteriorforkstars}]
The setup is similar to before. Take $\alpha$ with  $\dim(E)>\alpha>\alpha^{\circ}(n,k)$ and $\mu$ a Frostman measure on $E$. We split $E$ into separated compact sets $E_1,E_2, \dots ,E_k$ and $E_{k+1}$ with positive $\mu$ measure and satisfying transversality condition on $E_1, \dots, E_k$. As before, $E_1, E_2, \dots, E_k$ are the sets from which we draw the pins. Denote $\mu_i$ the restriction of $\mu$ to $E_i$, $1\leq i\leq k+1$, and $\bm{\mu}=\prod_{i=1}^{k} \mu_i$.

Recall that in the proof of \cref{lemma: equalityreplacingliuidentity}, for $\bm{x}$ fixed, we showed that the Fourier transform of 
\[\bm{s} \mapsto \mc{A}_{\Gamma_{\bm{x}}(\bm{s})}f(\bm{x})J_{\bm{x}}(\bm{s})1_{(-1/4,\infty)^k}(\bm{s})\]
is 
\[\bm{r}\mapsto \mc{E}f(-(r_1\pi^{-1}(x_1)+\dots + r_k\pi^{-1}(x_k))).\]

If we show that the integral
\begin{equation}\label{eq:estimatefornonemptyinteriorkstar}
 \int\int_{\R^k}|\mc{E}(\mu_{k+1})(r_1\pi^{-1}(x_1)+\dots +r_k\pi^{-1}(x_k))|^2 \, (1+|\bm{r}|)^{k+\epsilon}d\bm{r} \, d\bm{\mu}(\bm{x})   
\end{equation}
is finite for some $\epsilon>0$, then by Sobolev embedding one has that for $\prod_{i=1}^{k}\mu_i$ a.e. $(x_1,\dots ,x_k)\in E_1\times \dots \times E_k$. the function 
$$\bm{s}\mapsto \mathcal{A}_{\Gamma_{\bm{x}}(\bm{s})}\mu_{k+1}(\bm{x})J_{\bm{x}}(\bm{s})=(d_{\bm{x}})_{*}(\mu_{k+1})(\Gamma_{\bm{x}}(\bm{s}))J_{\bm{x}}(\bm{s})$$ is continuous. Since $\Gamma_{\bm{x}}:(-1/4,\infty)^k\rightarrow (0,\infty)^{k}$ is a diffeomorphism, continuity is also true for the map 
$$\bm{t}\mapsto (d_{\bm{x}})_{*}(\mu_{k+1})(\bm{t})W(\bm{x},\bm{t}).$$
The factor $W(\bm{x},\bm{t})=\prod_{j=1}^{k} J(\Gamma_{x_j}^{-1}(t_j))=\frac{2^k|x_1|^2\dots |x_k|^2}{t_1t_2\dots t_k}$ is continuous positive function, so 
$$\bm{t}\mapsto (d_{\bm{x}})_{*}(\mu_{k+1})(\bm{t})\text{ is continuous}.$$

Therefore, one can find a ball in the support of $ \mapsto(d_{\bm{x}})_{*}(\mu_{k+1})$, that is $\Delta_{x_1,\dots,x_k}(E_{k+1})$ has nonempty interior. 

We are left to check the main estimate (\ref{eq:estimatefornonemptyinteriorkstar}). Let $\nu_i=(\pi^{-1})_*(\mu_i)$ and $F_i=\pi^{-1}(E_i)$ for $i=1,2,\dots, k$, and note that

\begin{align*}
        \int_{E_1\times \dots \times E_k}\int_{\R^k}&|\mc{E}(\mu_{k+1})(r_1\pi^{-1}(x_1)+\dots +r_k\pi^{-1}(x_k))|^2 \, (1+|\bm{r}|)^{k+\epsilon}d\bm{r} \, d\bm{\mu}(\bm{x})\\   =\int_{F_1\times \dots \times F_k}\int_{\R^k}&|\mc{E}(\mu_{k+1})(r_1y_1+\dots +r_ky_k)|^2 (1+|\bm{r}|)^{k + \epsilon} \, d\bm{r} \, d(\nu_1\times \dots \times \nu_k)(\bm{y})  
\end{align*}

Since $\bm{\nu}$ is a finite measure, we reduce to integrating over $|\bm{r}|>1$ instead. So it suffices to show that 
\[\int_{F_1\times \dots \times F_k}\int_{|\bm{r}|>1}|\mc{E}(\mu_{k+1})(r_1y_1+\dots +r_ky_k)|^2 |\bm{r}|^{k + \epsilon} \, d\bm{r} \, d\bm{\nu}(\bm{y}) < \infty\]
for some $\epsilon > 0$. The steps to show that this integral is finite are similar as for the positive measure case. 

First, we claim that to bound the factor of $\bm{r}$ inside the integral, we can use 
$$|\bm{r}|\lesssim |r_1y_1+r_2y_2+\dots +r_ky_k|.$$
Indeed, 
\[r_j y_j = r_j \pi^{-1}(x_j) = \Phi(x_j, r_j) \in \text{supp} \, w_j.\]
Therefore by \cref{lem:weight_support}, we have that
\[|\bm{r}| \lesssim |r_1| + ... + |r_j| \lesssim |r_1 y_1| + ... + |r_ky_k| \lesssim |r_1y_1+r_2y_2+\dots +r_ky_k|\]
for all \(y_j\in F_j\) and all \(\bm r\in\R^k\).

Denote
\[
w_j:=\Psi_*(\nu_j\times dr_j),\qquad \Psi(y,r)=ry,
\]
and let \(w=w_1*\cdots*w_k\). Following the same argument as in the proof of Theorem~\ref{thm:parab_extension}, we can write

\begin{align*}
    \int_{F_1\times \dots \times F_k}\int_{|\bm{r}|>1}&|\mc{E}(\mu_{k+1})(r_1y_1+\dots +r_ky_k)|^2 |\bm{r}|^{k + \epsilon} \, d\bm{r} \, d(\nu_1\times \dots \times \nu_k)(\bm{y})  \\
    \lesssim 
    & \int |\mc{E}(\mu_{k+1})(y_1 + \cdots + y_k)|^2 |y_1+\dots +y_k|^{k+\epsilon}\, dw_1(y_1)\dots dw_k(y_k) \\ = & \int_{\R^{n+1}} |\mc{E}(\mu_{k+1})(y)|^2 |y|^{k+\epsilon}\, dw(y).
\end{align*}

From this point, the estimates are the same as in the proof of Theorem \ref{thm:main}, except we have an additional power of $R^{k+\epsilon}$ in the local estimates. I.e., we get 
\[R_i^{-\alpha+k-1}\int_{A_{R_i}}|\mc{E}(\mu_{k+1})(y)|^2 |y|^{k+\epsilon}dw^i(y)  \lesssim R_i^{k -\alpha+k-1 +\frac{\alpha+1}{n+1}+\epsilon}\|(\mu_{k+1})_{\leq R_i}\|_{L^2}^2\]
To close the estimate, we will need
\[k-\alpha+k-1+\frac{\alpha+1}{n+1}<\alpha-n,\] that is 
\[\alpha>\frac{n+2k-1}{2}+\frac{1}{4}+\frac{4k+1}{4(2n+1)}.\]
\end{proof}

\par

The technique in the proof of the theorem above works for all $k\geq 1$. When $k=1$, we can do better, though. That is the content of Theorem \ref{thm: nonemptyinteriorfor1star}, whose proof is provided in Appendix A.

\begin{remark}
    The threshold we get for $n = 2$ and $k=1$ is trivial. In fact, the $L^2$ method strategy used above will never give something nontrivial for the nonempty interior problem in the plane. Indeed, if 
    \[{\mathcal{S}}f(x) := \sup_{1 \leq t \leq 2}|\mc{A}_tf(x)|,\]
    then an $L^2$ bound 
    \[\int_{B(0,1)}\int_{1}^{2}| \partial_{t}^{1/2+\epsilon}(d_x)_*f(t)|^2 \, dt \, dx\lesssim \|f\|_{L^2}^2\]
    would imply that ${\mathcal{S}}$ is bounded on $L^2$ by Sobolev's embedding. However, it is known that ${\mathcal{S}}$ is not bounded on $L^2$ in dimension $2$.
\end{remark}

As an application of Theorem \ref{thm: nonemptyinteriorforkstars}, we get Corollary \ref{cor:interiorforevennecklaces}, which brings some new information on the nonempty interior of pinned even necklaces.

\begin{proof}[Proof of Corollary \ref{cor:interiorforevennecklaces}]

First, let us discuss the case $l=2$ to get some intuition. Let $E\subset \R^n$ with $\dim(E)>\alpha^{\circ}(n,2)$. Take an $s$-Frostman measure on $E$ with $s>\alpha^{\circ}(n,2)$. Take $E_1,E_2,E_3,E_4$ separated pieces of $E$ (that is $|x_i-x_j|\gtrsim 1$ for all $i\neq j$ and $x_i\in E_i$) with $\mu(E_j)>0$ for all $j=1,2,3,4$, and with the transversality condition on $E_1$ and $E_3$. 

From our result for $2$-stars we know that for $\mu_{E_1}\times \mu_{E_3}$ almost everywhere $(x_1,x_3)\in E_1\times E_3$ one has that
$\Delta_{x_1,x_3}^{2-\text{star}}(E_2)$ has nonempty interior in $\R^2$. Moreover, the same is true for $E_2$ replaced by $E_4$. 

By intersecting a couple of subsets of $E_1\times E_3$ with zero measure complement with respect to the product measure $\mu_{E_1}\times \mu_{E_3}$, we get that for $\mu_{E_1}\times \mu_{E_3}$ almost everywhere $(x_1,x_3)\in E_1\times E_3$ both $ \Delta_{x_1,x_3}^{2-\text{star}}(E_2)$ 
    and $ \Delta_{x_1,x_3}^{2-\text{star}}(E_4)$ have nonempty interior in $\R^2$. For any such good pairs $(x_1,x_3)$, we note that up to a permutation of coordinates,
    \begin{align*}
      \Delta_{x_1,x_3}^{C_4}(E_2,E_4):=&\{(|x_1-x_2|,|x_2-x_3|,|x_3-x_4|,|x_4-x_1|)\colon x_2\in E_2,x_4\in E_4\}\\
      =&\Delta_{x_1,x_3}^{2-\text{star}}(E_2)\times \Delta_{x_1,x_3}^{2-\text{star}}(E_4).  
    \end{align*}
    Since a product of two sets of nonempty interior in $\R^2$ will contain a $4$-dimensional nontrivial open cube, we conclude that $ \Delta_{x_1,x_3}^{C_4,\,2\,\text{pins}}(E_2,E_4)$ (and so $\Delta_{x_1,x_3}^{C_4,\,2\,\text{pins}}(E)$) has nonempty interior in $\R^4$.

Next, let us move to the general case $l\geq 2$. Similar to the case $l=2$, if we have a result for a chain in $2l-2$ edges with a pin in every other vertex starting from the endpoint of the chain, we can easily ``glue" it with a $2$-star to get the pinned $2l$-cycle of interest. So we are reduced to proving a result about even chains of length $2m$, $m\geq 1$.
   \begin{center}
\begin{tikzpicture}[scale=1.05, every node/.style={font=\small}]
  \definecolor{pinmagenta}{RGB}{190,80,210}
  \tikzset{pin/.style={circle, fill=pinmagenta, inner sep=2pt},
           free/.style={circle, fill=black, inner sep=2pt},
           edgeblue/.style={line width=1.5pt, blue!75},
           edgegray/.style={line width=1.5pt, gray!70}}

  \coordinate (v1) at (0,0);
  \coordinate (v2) at (1.1,0);
  \coordinate (v3) at (2.2,0);
  \coordinate (vmm1) at (5.2,0);
  \coordinate (vm) at (6.3,0);
  \coordinate (vmp1) at (7.4,0);

  \coordinate (vmp2) at (0.55,1.05);
  \coordinate (vmp3) at (1.65,1.05);
  \coordinate (v2m) at (5.75,1.05);
  \coordinate (v2mp1) at (6.85,1.05);

  \draw[edgeblue] (v1)--(vmp2)--(v2)--(vmp3)--(v3);
  \draw[edgeblue, dotted, line width=1.2pt] (v3)--(vmm1);
  \draw[edgeblue] (vmm1)--(v2m)--(vm);
  \draw[edgegray] (vm)--(v2mp1)--(vmp1);

  \node[pin, label=below:{\textcolor{pinmagenta}{$v_1$}}] at (v1) {};
  \node[pin, label=below:{\textcolor{pinmagenta}{$v_2$}}] at (v2) {};
  \node[pin, label=below:{\textcolor{pinmagenta}{$v_3$}}] at (v3) {};
  \node[pin, label=below:{\textcolor{pinmagenta}{$v_{m-1}$}}] at (vmm1) {};
  \node[pin, label=below:{\textcolor{pinmagenta}{$v_m$}}] at (vm) {};
  \node[pin, label=below:{\textcolor{pinmagenta}{$v_{m+1}$}}] at (vmp1) {};

  \node[free, label=above:{$v_{m+2}$}] at (vmp2) {};
  \node[free, label=above:{$v_{m+3}$}] at (vmp3) {};
  \node[free, label=above:{$v_{2m}$}] at (v2m) {};
  \node[free, label=above:{$v_{2m+1}$}] at (v2mp1) {};
\end{tikzpicture}
\end{center}

More precisely, let $\text{Ch}_{2m}$ be a $2m$-chain ($2m$ edges). Pick the set of pins $\mathcal{P}=\{v_1,v_2,\ldots,v_{m+1}\}$
as in the picture; that is, every other vertex is being pinned.

We want to prove by induction that if
\[
  E\subset \R^n,\,  \dim(E)>\alpha^{\circ}(n,2)
\]
and $E_1,\ldots,E_{2m+1}$ are separated subsets of $E$ with
$\mu(E_i)>0$ for all $i$ such that every pair $(E_i,E_{i+1}),1\leq i\leq m$ is transverse, then for
\[
    \prod_{i=1}^{m+1}\mu_{E_i}\text{-a.e.} \,(x_1,x_2,\dots, x_{m+1})
\]
one has
\[
   \text{int}  \left(
\Delta^{\text{Ch}_{2m},P}_{x_1,\ldots,x_{m+1}}
        (E_{m+2},\ldots,E_{2m+1})
    \right)\neq \emptyset.
\]

We do this using induction on $2m$. If \(m=1\), this is the \(2\)-star case of Theorem~\ref{thm: nonemptyinteriorforkstars}. Assume $m\geq 2$ and the statement is true for $m-1$. Then let
\[
    G^{(m)}\subset E_1\times E_2\times\cdots\times E_m
\]
be a subset with product measure
\[
    \left(\prod_{i=1}^{m}\mu_{E_i}\right)\left(\prod_{i=1}^{m} E_i\setminus G^{(m)}\right)=0
\]
and such that, for all $(x_1,x_2,\ldots,x_m)\in G^{(m)}$, one has
\[
    \text{int}\left(
        \Delta^{\text{Ch}_{2m-2},\widetilde{\mc{P}}}_{x_1,\ldots,x_m}
    (E_{m+2},E_{m+3},\ldots,E_{2m})
    \right)\neq \emptyset,
\]
where $
    \widetilde{\mc{P}}=\{v_1,v_2,\ldots,v_m\}$. From the case $m=1$, one can also find a subset
\[
    G^{(2)}\subset E_m\times E_{m+1}, \text{ with }(\mu_{E_m}\times \mu_{E_{m+1}})((E_{m}\times E_{m+1})\setminus G^{(2)})=0
\]
such that, for all $(z_m,z_{m+1})\in G^{(2)}$, one has
\[
    \text{int}\left(
        \Delta^{2\text{-star}}_{z_m,z_{m+1}}
        (E_{2m+1})
    \right)\neq \emptyset.
\]

We claim that, if one defines
\[
    G^{(m+1)}:=\bigl\{(x_1,x_2,\ldots,x_{m+1})\in \prod_{i=1}^{m+1}E_i:
    (x_1,\ldots,x_m)\in G^{(m)}
    \text{ and }(x_m,x_{m+1})\in G^{(2)}\bigr\},
\]
then
\[
    \left(\prod_{i=1}^{m+1}\mu_{E_i}\right)\left(\prod_{i=1}^{m+1} E_i\setminus G^{(m+1)}\right)=0.
\]
Indeed,
\begin{align*}
&\left(\prod_{i=1}^{m+1}\mu_{E_i}\right)
\left(\left(\prod_{i=1}^{m+1}E_i\right)\setminus G^{(m+1)}\right) \\
&\qquad\leq
\left(\prod_{i=1}^{m+1}\mu_{E_i}\right)
\left((G^{(m)})^c\times E_{m+1}\right) \\
&\qquad\quad+
\left(\prod_{i=1}^{m+1}\mu_{E_i}\right)
\left(\left(\prod_{i=1}^{m-1}E_i\right)\times (G^{(2)})^c\right)=0. 
\end{align*}

For every $(x_1,\ldots,x_{m+1})\in G^{(m+1)}$, the pinned distance set associated to $\mathrm{Ch}_{2m}$ can be identified, up to a permutation of coordinates, with the product of the pinned distance set associated to $\mathrm{Ch}_{2m-2}$ and the pinned distance set associated to the final $2$-star. Since both factors have nonempty interior, and products of open sets are open, it follows that
\[
\Delta^{\mathrm{Ch}_{2m},P}_{x_1,\ldots,x_{m+1}}
(E_{m+2},\ldots,E_{2m+1})
\]
has nonempty interior in $\R^{2m}$. This completes the induction. Applying the resulting chain statement to the decomposition of $C_{2l}$ described above yields the desired conclusion for even cycles.
\end{proof}

Next, we move to dimension estimates for $k$-stars. Given a set $E\subset \R^n$ whose dimension is not large enough to apply Theorem \ref{thm:main}, that is $\dim(E)\leq \alpha_{+}(n,k)$, we investigate what dimensional lower bounds one can still get.

\begin{theorem}[Dimension lower bound for \(k\)-stars]\label{thm:kstar_dimension_lower}
Assume $1\leq k<n$. let \(E\subset \R^n\) be compact with
\[
\max\left\{\frac{n^2}{2n+1},k-1\right\}<\dim(E)\leq \alpha_{+}(n,k)
=
\frac{n^2+nk+k}{2n+1}.
\]
Set
\[
\delta(E):=
\frac{2n+1}{n+1}\dim(E)-\frac{n^2}{n+1}.
\]
Then, for every sufficiently small \(\epsilon>0\), there exists a Frostman measure
\(\mu_\epsilon\) on \(E\) such that there is an abundance of pins
\((x_1,\dots,x_k)\in E^k\) with respect to \(\mu_\epsilon\) satisfying
\[
\dim\bigl(\Delta^{k\text{-star}}_{x_1,\dots,x_k}(E)\bigr)
\geq
\delta(E)-\epsilon.
\]
\end{theorem}

\begin{proof}
 Let $\mu_\epsilon$ be an $\alpha$-Frostman measure on $E$ with $\alpha=\dim(E)-c\epsilon$ with $c>0$ to be chosen. Following the approach in the proof of \cref{thm: nonemptyinteriorforkstars}, we will check for which values of \(0\leq \beta<k/2\) one has that the measure
\[\bm{s} \mapsto \mc{A}_{\Gamma_{\bm{x}}(\bm{s})}\mu_{k+1}(\bm{x})J_{\bm{x}}(\bm{s})1_{(-1/4,\infty)^k}(\bm{s}) = (d_{\bm{x}})_*(\mu_{k+1})(\Gamma_{\bm{x}}(\bm{s}))J_{\bm{x}}(\bm{s})1_{(-1/4,\infty)^k}(\bm{s})\]
has support of dimension at least $k-2\beta$ in $\R^k$, for $\mu_{E_1}\times \dots \times \mu_{E_k}$ a.e. $\bm{x}$.

Then for such $\beta$,
\[
\dim(\Gamma_{\bm{x}}^{-1}\bigl(\Delta^{k\text{-star}}_{\bm{x}}(E_{k+1})\bigr))\ge k-2\beta
\quad\text{for }\prod_{i=1}^k\mu_{E_i}\text{-a.e. }\bm{x}=(x_1, \dots, x_k),
\]
and so
\[
\dim\bigl(\Delta^{k\text{-star}}_{\bm{x}}(E_{k+1})\bigr)\ge k-2\beta
\quad\text{for }\prod_{i=1}^k\mu_{E_i}\text{-a.e. }\bm{x}=(x_1, \dots, x_k).
\]

By the standard energy integral criterion and by taking the Fourier transform of the measure above, the question becomes for which $0\leq \beta<k/2$ one has 
\begin{equation}\label{eqn:hausdorff_dim_integral}
    \int\int_{\R^k}|\mc{E}(\mu_{k+1})(r_1\pi^{-1}(x_1)+\dots +r_k\pi^{-1}(x_k))|^2 \, |\bm{r}|^{-2\beta}d\bm{r} \, d\bm{\mu}(\bm{x}) < \infty.
\end{equation}

After Plancherel, \cref{eqn:hausdorff_dim_integral} reduces to estimating

\begin{align*}
&\int_{|\bm{r}|>1}\int |\mathcal E(\mu_{k+1})(r_1y_1+\cdots+r_ky_k)|^2|\bm{r}|^{-2\beta}\,d\bm\nu(\bm y)d\bm{r}\\
\lesssim &\sum_i R_i^{-\alpha+k-1}\int_{A_{R_i}}|\mathcal E(\mu_{k+1})(y)|^2|y|^{-2\beta}dw^i(y)\\
\lesssim &\sum_{i\ge0}R_i^{-\alpha+k-1-2\beta}\int_{A_{R_i}}|\mathcal E(\mu_{k+1})(y)|^2dw^i(y)\\
\lesssim_{\epsilon} & \sum_i R_i^{-\alpha+k-1-2\beta}
\Bigl(R_i^{\frac{\alpha+1}{n+1}+\epsilon}\|(\mu_{k+1})_{\leq R_i}\|_{L^2}^2\Bigr).
\end{align*}
For finiteness of the sum above, it is enough to choose \(\beta\) satisfying
\[
-\alpha+k-1-2\beta+\frac{\alpha+1}{n+1}<\alpha-n,
\]
that is,
\[
\,n+k-1+\frac1{n+1}-\alpha\left(\frac{2n+\text{1}}{n+1}\right)<2\beta.
\]
or equivalently,
$$k-2\beta<\frac{2n+1}{n+1}\alpha-\frac{n^2}{n+1}=\delta(E)-\frac{c\epsilon(2n+1)}{n+1}.$$
Choose $c$ sufficiently small such that \(\delta(E)-\epsilon<\delta(E)-c\frac{\epsilon(2n+1)}{n+1}\), then we may choose \(\beta\) satisfying the condition above and such that
\[
k-2\beta\geq \delta(E)-\epsilon.
\]
The statement is then true as long as $\varepsilon$ is sufficiently small.

\end{proof}

As an application of the dimensional estimates above for a $2$-star, we prove the following dimensional lower bounds for pinned triangles.

\begin{corollary}\label{cor: dim triangle}
Let $E\subset \R^n$ with $\alpha_{+}(n,1)<\dim(E)<\alpha_{+}(n,2)$. Consider 
$$\Delta^{triangle}_{x_1}(E)=\{(|x_1-x_2|,|x_2-x_3|,|x_3-x_1|)\colon x_2,x_3\in E\}\subset \R^3.$$
Then for each $\epsilon>0$, there exists $x_1\in E$ such that 
$$\dim(\Delta^{triangle}_{x_1}(E))\geq 1+\frac{2n+1}{n+1}\dim(E)-\frac{n^2}{n+1}-\epsilon.$$
\end{corollary}

   \begin{proof}
Let \(E_1,E_2,E_3\) be compact separated sets equipped with measures \(\mu_{E_i}\) be as before, and $E_1$ and $E_2$ transverse (we remark that $\dim(E)>\alpha_{+}(n,1)=\frac{n^2+n+1}{2n+1}>\max\{k-1,\frac{n^2}{2n+1}\}$ for $k=2$).

Start by shrinking \(E_1\) to a compact subset \(E_1' \subset E_1\) such that
\[
\mu_{E_1}(E_1')>0
\]
and
\[
(d_{x_1})_{*}(\mu_{E_2})\in L^2(\R),
\qquad \forall x_1\in E_1'.
\]
That is possible since
\[
\dim E > \alpha_{+}(n,1).
\]

Applying what we know for \(2\)-stars, we can find a subset
\[
G^{(2)} \subset E_1' \times E_2
\]
such that
\[
\mu_{E_1} \times \mu_{E_2}(G^{(2)})>0
\]
and, for all \((x_1,x_2)\in G^{(2)}\),
\[
\dim\bigl(\Delta^{2\text{-star}}_{x_1,x_2}(E_3)\bigr)
\geq \gamma
:= \frac{2n+1}{n+1}\dim E - \frac{n^2}{n+1}-\epsilon.
\]

Since \(\mu_{E_1}\times \mu_{E_2}(G^{(2)})>0\), there exists \(E_1''\subset E_1\), compact,
such that
\[
\mu_{E_1}(E_1'')>0
\]
and
\[
\mu_2\bigl(E_2(x_1):=\{x_2\in E_2:(x_1,x_2)\in G^{(2)}\}\bigr)
\gtrsim c\,\mu(E_2)
\qquad \forall x_1\in E_1''.
\]

From the \(L^2\) assumption
\[
(d^{x_1})_{*}(\mu_{E_2})\in L^2,
\]
it follows that we still have
\[
\mathcal{L}^1\bigl(\Delta_{x_1}(E_2(x_1))\bigr)>0.
\]

We claim that
\[
\dim\bigl(\Delta^{triangle}_{x_1}(E)\bigr)\ge 1+\gamma
\qquad \forall x_1\in E_1''.
\]

Indeed, by Theorem A.2 in \cite{OT2020}, for all \(\varepsilon>0\) sufficiently small,
\[
\mathcal{H}^{1+\gamma-\varepsilon}
\bigl(\Delta^{triangle}_{x_1}(E)\bigr)
\gtrsim
c\mathcal{H}^1\bigl(\Delta_{x_1}(E_2(x_1))\bigr)>0
\]
as long as we know that, for all \(t_1\in \Delta_{x_1}(E_2(x_1))\),
\[
\mathcal{H}^{\gamma-\varepsilon}
\bigl(\{(t_1,t_2,t_3):(t_1,t_2,t_3)\in
\Delta^{triangle}_{x_1}(E)\}\bigr)
\geq c.
\]

To check this, take \(x_2(t_1)\in E_2(x_1)\) such that
\[
|x_1-x_2(t_1)|=t_1.
\]
Then
\[
\mathcal{H}^{\gamma-\varepsilon}
\bigl(\{(t_1,t_2,t_3):(t_1,t_2,t_3)\in
\Delta^{triangle}_{x_1}(E)\}\bigr)
\ge
\mathcal{H}^{\gamma-\varepsilon}
\bigl(\Delta^{2\text{-star}}_{x_1,x_2(t_1)}(E_3)\bigr).
\]

Since
\[
\dim\bigl(\Delta^{2\text{-star}}_{x_1,x_2(t_1)}(E_3)\bigr)\geq\gamma,
\]
the proof follows.
\end{proof}

\appendix

\section{Improvement of pinned nonempty interior for distances}\label{appendixA:nonemptyviaDZ}

In this appendix, we observe that estimates of Du--Zhang \cite{DZ2019} yield improved thresholds for pinned nonempty interior of distance sets when \(n\geq 4\), improving the best known thresholds from \cite{BFOP2026} for $n\geq 4$. That is the content of Theorem \ref{thm: nonemptyinteriorfor1star}, which we now prove.

\begin{proof}[Proof of Theorem \ref{thm: nonemptyinteriorfor1star}]
    We provide a sketch of the proof since most of the arguments are standard. It is enough to show that for an $\alpha$-Frostman measure $\mu_1$ in $B^n(0,1)$ with support $E_1$, there exists $\epsilon>0$ such that for all smooth functions $f$ supported on $B^n(0, 1)$ whose support is a positive distance from $E_1$, one has 
    \[\int_{E_1}\int_{1}^{2}| \partial_{t}^{1/2+\epsilon}\mc{A}_tf(x_1)|^2 \, dt \, d\mu_1(x_1)\lesssim \int_{\R^n} |\xi|^{\alpha - n - \epsilon}|\widehat{f}(\xi)|^2 \, d\xi\]
    if 
    \[\alpha > \frac{n(n + 1)}{2n - 1} = \frac{n}{2} + \frac{3}{4} + \frac{3}{4(2n-1)}.\]
    Without loss of generality, assume that $\text{dist}(\text{supp} \, f, E_1) > 5/4$. For $j \geq 1$, let $\mc{P}_j$ be the Littlewood-Paley projection onto frequencies $\approx 2^j$ and for $j = 0$, let $\mc{P}_0$ be the projection onto frequencies in $B^n(0, 1)$. Write $f_j := \mc{P}_jf$. \par
    We first split the left-hand side into the Littlewood-Paley pieces:
    \[\left(\int_{E_1}\int_{1}^{2}| \partial_{t}^{1/2+\epsilon}\mc{A}_tf(x_1)|^2dt \, d\mu_1(x_1)\right)^{1/2}  \leq \sum_{j \geq 0}\left(\int_{E_1}\int_{1}^{2}| \partial_{t}^{1/2+\epsilon}\mc{A}_tf_j(x_1)|^2dt \, d\mu_1(x_1)\right)^{1/2}.\]
    The idea is that since $f_j$ has frequencies in $\{|\xi| \approx 2^j\}$, then so does the function $t \mapsto \mc{A}_tf_j(x_1)$. We will take the Fourier transform in $(x_1, t)$ of $(x_1, t) \mapsto \mc{A}_tf(x_1)$. In the \(x_1\)-variable, since \(\mc A_t f=f*\sigma_t\) with \(\sigma_t\) denoting unnormalized surface measure on \(tS^{n-1}\), we have
\[
\widehat{\mc A_t f}(\xi)
=
t^{n-1}\widehat{\sigma}(t\xi)\widehat f(\xi).
\]
Since \(t\in[1,2]\), the factor \(t^{n-1}\) is harmless in the estimates below.

    Expressing $\widehat{\sigma}$ in terms of Bessel functions gives 
    \[\widehat{\sigma}(t\xi) = C \int_{-1}^{1}e^{it|\xi|s}(1 - s^2)^{(n - 3) / 2} \, ds.\]
    
    Hence, the Fourier transform in $t$ of $\widehat{\sigma}(t\xi)$ is
    \[C (1 - (s / |\xi|)^2)_+^{(n - 3)/2}|\xi|^{-1} = C (|\xi|^2 - s^2)_{+}^{(n - 3) / 2}|\xi|^{-n + 2}.\]
    
    Therefore the Fourier transform in $(x_1, t)$ of $\mc{A}_tf_j(x_1)$ is
    \[C (|\xi|^2 - s^2)_{+}^{(n - 3) / 2}|\xi|^{-n + 2}\widehat{f_j}(\xi).\]
    By inverting the Fourier transform in $\xi$, the Fourier transform of $t \mapsto \mc{A}_tf_j(x)$ is 
    \[s \mapsto C \int_{|\xi| \geq s} (|\xi|^2 - s^2)_{+}^{(n - 3) / 2}|\xi|^{-n + 2}\widehat{f_j}(\xi) e^{2\pi i x \cdot \xi} \, d\xi.\]
    This function is nonzero only if $s \leq 2^{j + 1}$. Therefore by applying Plancherel in $t$, we can replace the derivative with $2^{j(\frac{1}{2} + \epsilon)}$
    \[\int_{E_1}\int_{1}^{2}| \partial_{t}^{1/2+\epsilon}\mc{A}_tf_j(x_1)|^2 \, dt \, d\mu_1(x_1) \lesssim 2^{j(1+2\epsilon)}\int_{E_1}\int |\mc{A}_tf_j(x_1)|^2 \, dt \, d\mu_1(x_1).\]
    Since $f_j$ is essentially supported in the $2^{-j}$-neighborhood of $\text{supp} \, f$, which has positive distance from $E_1$, we can assume the integration in $t$ is over the interval $[1, 2]$. 
    Since \(t\in[1,2]\), the normalized and unnormalized spherical averages differ by harmless constants. By applying Liu's identity and Du-Zhang \cite{DZ2019}, we have that
    \begin{align*}
        \int_{E_1}\int_{1}^{2}| \mc{A}_tf_j(x_1)|^2 \, dt \, d\mu_1(x_1) & \approx \int_{E_1}\int_{0}^{\infty}|f_j * \widehat{\sigma}_r(x_1)|^2 r^{n - 1} \, dr \, d\mu_1(x_1) \\
        & = \int_{r \approx 2^{j}}\int_{E_1}|f_j * \widehat{\sigma}_r(x_1)|^2 d\mu_1(x_1) \, r^{n - 1} \, dr \\ 
        & \lesssim \int_{r \approx 2^{j}}\|\widehat{f}_j\|_{L^2(rS^{n - 1})}^2r^{-\alpha + \frac{\alpha}{n}} \, r^{n - 1} \, dr \\ 
        & \approx \int_{|\xi| \approx 2^j}|\widehat{f}(\xi)|^2|\xi| ^{-\alpha + \frac{\alpha}{n}} \, d\xi.
    \end{align*}
    Therefore 
    \begin{align*}
        \int_{E_1}\int_{1}^{2}| \partial_{t}^{1/2+\epsilon}\mc{A}_tf_j(x_1)|^2 \, dt \, d\mu_1(x_1) &  \lesssim \int_{|\xi| \approx 2^j}|\widehat{f}(\xi)|^2|\xi|^{-\alpha + \frac{\alpha}{n} + 1 + 2\epsilon} \, d\xi \\
        & \lesssim 2^{-\epsilon j}\int|\widehat{f}(\xi)|^2|\xi|^{-\alpha + \frac{\alpha}{n} + 1 + 3\epsilon} \, d\xi \\
        & \lesssim 2^{-\epsilon j}\int|\widehat{f}(\xi)|^2|\xi|^{\alpha - n - \epsilon} \, d\xi \\
    \end{align*}
    since 
    \[-\alpha + \frac{\alpha}{n} + 1 + 4\epsilon < \alpha - n\] if $\alpha > n(n + 1) / (2n - 1)$ and $\epsilon$ is small enough. Taking square roots and summing in \(j\) gives the desired estimate.

    Applying this estimate to smooth approximations of a Frostman measure supported on the unpinned set and passing to the limit gives that the corresponding pinned distance measure has a density in \(H^{1/2+\epsilon}(\R)\), hence a continuous representative. Since the measure is nonzero, its support contains a nontrivial interval.
\end{proof}

\section{Application to Discrete Variants of \texorpdfstring{$k$}{k}-stars}\label{appendixB:discrete}

Inspired by the proof of Corollary 1.5 in \cite{GIOW20}, we will prove the following corollary concerning finite point sets in $\mathbb{R}^n$. To our knowledge, this seems to be the first $k$-star result in the discrete setting for $k\geq 2$ and may be of independent interest. 

\begin{corollary}[Discrete pinned \(k\)-stars]
Let \(1\le k<n\), and let \(P\subset [0,1]^n\) be a set of \(N\) points with mutual separation
\[
|p-p'|\gtrsim N^{-1/n},\qquad p\neq p'.
\]
Then for every \(\varepsilon>0\), there exist pins \(x_1,\dots,x_k\in P\) such that
\[
\bigl|\Delta^{k\text{-star}}_{x_1,\dots,x_k}(P)\bigr|
\gtrsim_\varepsilon
N^{\frac{k}{\alpha_+(n,k)}-\varepsilon},
\]
where
\[
\alpha_+(n,k)=\frac{n^2+nk+k}{2n+1}.
\]
\end{corollary}

\begin{proof}
Fix \(\varepsilon>0\). Choose \(s\) with \(\alpha_+(n,k)<s<n\) and sufficiently close to \(\alpha_+(n,k)\) so that
\[
\frac{k}{s}>\frac{k}{\alpha_+(n,k)}-\varepsilon.
\]
Define the \(N^{-1/s}\)-thickening of \(P\) by
\[
P_{N^{-1/s}}:=\bigcup_{p\in P}B(p,N^{-1/s}),
\]
and define the normalized measure
\[
d\mu_P^s(x)
=
N^{-1}N^{n/s}
\sum_{p\in P}\chi_{B(p,N^{-1/s})}(x)\,dx.
\]

We claim that \(\mu_P^s\) is an \(s\)-Frostman measure with constant independent of \(N\). Indeed, let \(B(x,r)\subset\R^n\). There are a few different possibilities for $r>0$.

If \(0<r\leq N^{-1/s}\), then 
\[
\mu_P^s(B(x,r))
\lesssim
N^{-1}N^{n/s}r^n
=
(N^{-1/s})^{s-n}r^n
\leq r^s,
\]
since \(r\leq N^{-1/s}\) and \(s<n\).

If \(N^{-1/s}\leq r\leq N^{-1/n}\), then \(B(x,r)\) intersects only \(O(1)\) of the balls \(B(p,N^{-1/s})\), so
\[
\mu_P^s(B(x,r))\lesssim N^{-1}\leq r^s.
\]

If \(N^{-1/n}\leq r\leq 1\), then the \(N^{-1/n}\)-separation of \(P\) implies that \(B(x,r)\) intersects at most \(\lesssim Nr^n\) of the balls \(B(p,N^{-1/s})\). Hence
\[
\mu_P^s(B(x,r))
\lesssim
N^{-1}(Nr^n)
=
r^n
\leq r^s,
\]
since \(r\leq1\) and \(s<n\). Finally, if \(r\geq1\), then
\[
\mu_P^s(B(x,r))\leq\mu_P^s(\R^n)\lesssim1\leq r^s.
\]

Thus \(\mu_P^s\) is an \(s\)-Frostman measure with constant independent of \(N\). Since \(s>\alpha_{+}(n,k)\) and the Frostman constant of \(\mu_P^s\) is uniform in \(N\), the quantitative proof of Theorem~\ref{thm:main} applies to \(P_{N^{-1/s}}\). Hence there exist pins
\[
\bm{x}_0=(x_{0,1},\dots,x_{0,k})\in (P_{N^{-1/s}})^k
\]
such that
\[
\mathcal L^k\bigl(
\Delta^{k\text{-star}}_{\bm{x}_0}(P_{N^{-1/s}})
\bigr)\gtrsim_s 1,
\]
where the implicit constant is independent of \(N\).

For each \(1\le i\le k\), choose \(x_i\in P\) with
\[
|x_i-x_{0,i}|\le N^{-1/s}.
\]
Let \(\bm{x}=(x_1,\dots,x_k)\). If \(y\in P_{N^{-1/s}}\), then there exists \(p\in P\) such that
\[
|y-p|\le N^{-1/s}.
\]

Hence, by the reverse triangle inequality,
\[
\bigl||x_{0,i}-y|-|x_i-p|\bigr|
\le
|(x_{0,i}-y)-(x_i-p)|.
\]
Therefore
\[
\bigl||x_{0,i}-y|-|x_i-p|\bigr|
\le
|(x_{0,i}-x_i)-(y-p)|
\le
|x_{0,i}-x_i|+|y-p|
\lesssim N^{-1/s}.
\]
Thus every coordinate of the \(k\)-star vector changes by at most \(CN^{-1/s}\), and consequently
\[
\Delta^{k\text{-star}}_{\bm{x}_0}(P_{N^{-1/s}})
\subset
\mathcal N_{C N^{-1/s}}
\bigl(
\Delta^{k\text{-star}}_{\bm{x}}(P)
\bigr)
\]
for some absolute constant \(C>0\).

Since \(\Delta^{k\text{-star}}_{\bm{x}}(P)\subset\mathbb R^k\) is finite, its \(C N^{-1/s}\)-neighborhood has \(k\)-dimensional Lebesgue measure bounded by
\[
\mathcal L^k\Bigl(
\mathcal N_{C N^{-1/s}}
\bigl(
\Delta^{k\text{-star}}_{\bm{x}}(P)
\bigr)
\Bigr)
\lesssim
N^{-k/s}
\bigl|\Delta^{k\text{-star}}_{\bm{x}}(P)\bigr|.
\]
Combining this with the lower bound above gives
\[
1
\lesssim
N^{-k/s}
\bigl|\Delta^{k\text{-star}}_{\bm{x}}(P)\bigr|.
\]
Thus
\[
\bigl|\Delta^{k\text{-star}}_{\bm{x}}(P)\bigr|
\gtrsim
N^{k/s}.
\]
By our choice of \(s\),
\[
N^{k/s}
\ge
N^{\frac{k}{\alpha_+(n,k)}-\varepsilon}.
\]
Hence
\[
\bigl|\Delta^{k\text{-star}}_{x_1,\dots,x_k}(P)\bigr|
\gtrsim_\varepsilon
N^{\frac{k}{\alpha_+(n,k)}-\varepsilon}.
\]
\end{proof}

\bibliographystyle{amsalpha}
\bibliography{./ref}

\end{document}